\documentclass{article}

\usepackage
{
    graphicx,
    amssymb,
    amsmath,
    amsthm,
    dsfont, 
    xcolor,
    epstopdf,
    braket,
    bm,
    bbm,
    enumerate,
    authblk
}

\usepackage[colorlinks,
            pdffitwindow=false,
            plainpages=false,
            pdfpagelabels=true,
            pdfpagemode=UseOutlines,
            pdfpagelayout=SinglePage,
            bookmarks=false,
            colorlinks=true,
            hyperfootnotes=false,
            linkcolor=blue,
            citecolor=green!50!black]{hyperref}

\usepackage[hmargin=2cm,vmargin=2.5cm]{geometry}
\usepackage[bf]{caption}

\DeclareMathAlphabet{\mathpzc}{OT1}{pzc}{m}{it}

\newcommand{\subfiguretitle}[1]{{\scriptsize{#1}} \\[1mm]}
\newcommand{\R}{\mathbb{R}}                                     
\newcommand{\C}{\mathbb{C}}                                     
\newcommand{\degree}{\,^{\circ}}                                
\newcommand{\verteq}{\rotatebox{90}{$\,=$}}                     
\newcommand{\unit}[1]{\ensuremath{\, \mathrm{#1}}}              
\newcommand{\textsub}[1]{\text{\tiny{#1}}}                      
\newcommand{\innerprod}[2]{\left\langle #1,\, #2 \right\rangle} 
\providecommand{\norm}[1]{\left\lVert #1 \right\rVert}          
\providecommand{\grad}{\nabla}                                  
\renewcommand{\P}{\mathcal{P}}                                  
\newcommand{\K}{\mathcal{K}}                                    
\newcommand{\Sp}{\mathcal{S}}                                   
\newcommand*{\expect}{\mathsf{E}}                               
\newcommand*{\prob}{\mathsf{P}}                                 
\newcommand{\qstate}{\mathbb{Q}}                                
\newcommand{\pstate}{\mathbb{P}}                                

\newcommand\xqed[1]{\leavevmode\unskip\penalty9999 \hbox{}\nobreak\hfill \quad\hbox{#1}}
\newcommand{\exampleSymbol}{\xqed{$\triangle$}}

\DeclareMathOperator{\mspan}{span}

\newtheorem{theorem}{Theorem}[section]

\newtheorem{lemma}[theorem]{Lemma}

\newtheorem{definition}[theorem]{Definition}
\theoremstyle{definition}
\newtheorem{example}[theorem]{Example}
\newtheorem{remark}[theorem]{Remark}

\makeatletter
\renewcommand*\env@matrix[1][*\c@MaxMatrixCols c]{%
  \hskip -\arraycolsep
  \let\@ifnextchar\new@ifnextchar
  \array{#1}}
\makeatother

\title{On the numerical approximation of the \\ Perron--Frobenius and Koopman operator}
\author[1]{Stefan Klus}
\author[1]{P\'eter Koltai}
\author[1,2]{Christof Sch\"utte}
\affil[1]{Department of Mathematics and Computer Science, Freie Universit\"at Berlin, Germany}
\affil[2]{Zuse Institute Berlin, Germany}
\date{}

\begin{document}
\maketitle

\begin{abstract}
Information about the behavior of dynamical systems can often be obtained by analyzing the eigenvalues and corresponding eigenfunctions of linear operators associated with a dynamical system. Examples of such operators are the Perron--Frobenius and the Koopman operator. In this paper, we will review different methods that have been developed over the last decades to compute finite-dimensional approximations of these infinite-dimensional operators -- in particular Ulam's method and Extended Dynamic Mode Decomposition (EDMD) -- and highlight the similarities and differences between these approaches. The results will be illustrated using simple stochastic differential equations and molecular dynamics examples.
\end{abstract}

\section{Introduction}
\label{sec:Introduction}

The two main candidates for analyzing a dynamical system using operator-based approaches are the Perron--Frobenius and the Koopman operator. These two operators are adjoint to each other in appropriately defined function spaces and it should therefore theoretically not matter which one is used to study the system's behavior. Nevertheless, different methods have been developed for the numerical approximation of these two operators.

The Perron--Frobenius operator has been used extensively in the past to analyze the global behavior of dynamical systems stemming from a plethora of different areas such as molecular dynamics~\cite{PDHSM04,SS13}, fluid dynamics~\cite{FSvS14,FGTW15}, meteorology and atmospheric sciences~\cite{TvdBD15,TaLuLuDi15}, or engineering~\cite{VMS10,ober2015multiobjective}. Toolboxes for computing almost invariant sets or metastable states are available and efficiently approximate the system's behavior using adaptive box discretizations of the state space. An example of such a toolbox is GAIO~\cite{DFJ00}. This approach is, however, typically limited to low-dimensional problems.

Recently, several papers have been published focusing on data-based numerical methods to approximate the Koopman operator and to analyze the associated Koopman eigenvalues, eigenfunctions, and modes~\cite{BMM12, WKR14, WRK14}. These methods extract the relevant global behavior of dynamical systems and can, for example, be used to find lower-dimensional approximations of a system and to split a system into fast and slow subsystems as described in~\cite{FGH14a}. In many applications, the complex behavior of a dynamical system can be replicated by a small number of modes~\cite{WRK14}.

The approximation of the Perron--Frobenius operator typically requires short simulations for a large number of different initial conditions, which, without prior knowledge about the system, grows exponentially with the number of dimensions; the approximation of the Koopman operator, on the other hand, relies on potentially fewer, but longer simulations~\cite{BMM12}. However, we will show that this is not necessarily the case, the Perron--Frobenius operator can also be approximated using just a small number of long simulations. Thus, the latter approach might be well-suited for experimentally obtained data running just a few tests with different initial conditions for a longer time. Whether the numerically obtained operator then captures the full dynamics of the system, however, depends strongly on the initial conditions chosen.

While the Koopman operator is the adjoint of the Perron--Frobenius operator, the connections between different approaches to approximate these operators have -- to our knowledge -- not been fully described. In this paper, we will review different numerical methods to approximate the Perron--Frobenius operator and the Koopman operator and illustrate the similarities and differences between these approaches. We will mainly focus on simple stochastic differential equations and molecular dynamics applications.

The outline of this paper is as follows: In Section~\ref{sec:Transfer operators}, we will introduce the Perron--Frobenius operator and the Koopman operator and give a short description of basic properties. In Section~\ref{sec:Numerical approximation}, we will describe numerical methods (more precisely, generalized Galerkin methods) to obtain finite-dimensional representations of these operators and their eigenfunctions. Section~\ref{sec:Duality} illustrates the relationship between numerical methods developed for analyzing these operators. Section~\ref{sec:Examples} contains examples demonstrating the efficiency and characteristic properties of these numerical methods. A conclusion and possible future work will be outlined in Section~\ref{sec:Conclusion}. In Appendix~\ref{app:EDMD spatial essential}, we draw a connection between (i) extended dynamic mode decomposition applied to molecular dynamics simulation data and (ii) a special transfer operator used in molecular conformation analysis.

\section{Transfer operators}
\label{sec:Transfer operators}

\subsection{Perron--Frobenius operator}

\paragraph{Deterministic systems.}
Historically, transfer operators have been introduced in the field of \emph{ergodic theory}, where the main focus is on a \emph{measure-theoretic} characterization of the behavior of dynamical systems~\cite{Ko31,Hal56,Sin59,Orn70,LaMa94,BoGo12}. Due to this, the starting point of the considerations is often a \emph{measure space}~$(\mathbb{X},\mathfrak{B},\mu)$, a three-tuple of a \emph{state space}, a sigma-algebra, and a (probability) measure, respectively. The evolution of the state, usually in time, is described by a dynamical system~$\Phi : \mathbb{X} \to \mathbb{X}$, where~$\Phi$ is a $\mu$-measurable map. When not stated explicitly otherwise, time is considered to be discrete, hence a state~$x\in\mathbb{X}$ evolves as $ \{x, \Phi(x), \Phi^2(x), \ldots \}$. Nevertheless, most of the concepts carry over in a straightforward fashion to continuous-time systems, which we denote by~$\Phi^t$, $t \ge 0$.

In order to describe the statistical behavior of the dynamical system, we are interested in how~$\Phi$ affects distributions over state space. To this end, let us think of $f\in L^1(\mathbb{X}) := L^1(\mathbb{X},\mathfrak{B},\mu)$, with~$f\ge 0$ almost everywhere (a.e.) and~$\norm{f}_{L^1} = 1$, as the density of an $\mathbb{X}$-valued random variable~$\bm x$, we write $\bm x\sim f$. We wish to characterize the distribution of~$\Phi(\bm x)$. It turns out that if~$\Phi$ is \emph{non-singular}\footnote{$\Phi$ is (measure-theoretically) non-singular with respect to~$\mu$ if $\mu\circ\Phi^{-1}\ll \mu$; i.e., $\mu\circ\Phi^{-1}$ is absolutely continuous with respect to~$\mu$. This condition ensures that~$\Phi$ does not map sets of positive measure to sets of zero measure, that is, it can not destroy probability measure.} with respect to~$\mu$, then there is a $g\in L^1(\mathbb{X})$ such that~$\Phi(\bm x)\sim g$, and $\int_{\mathbb{A}} g\,d\mu = \int_{\Phi^{-1}(\mathbb{A})}f\,d\mu$ for all $\mathbb{A}\in\mathfrak{B}$. The mapping $f\mapsto g$ can be linearly extended to a linear operator $\P: L^1(\mathbb{X})\to L^1(\mathbb{X})$,
\begin{equation*}
    \int_{\mathbb{A}} \P f\,d\mu = \int_{\Phi^{-1}(\mathbb{A})}f\,d\mu, \quad \mathbb{A}\in\mathfrak{B},
\end{equation*}
the so-called \emph{Perron--Frobenius operator}~\cite{LaMa94,BoGo12}. It is a linear, \emph{positive} (i.e., $ f \ge 0 $ implies $ \P f\ge 0 $), \emph{non-expansive} (i.e., $\norm{\P f}_{L^1} \le \norm{f}_{L^1}$) operator, hence a \emph{Markov operator}. In addition, if the underlying measure~$\mu$ is \emph{invariant}, i.e., $\mu\circ \Phi^{-1} = \mu$, then $\P: L^p(\mathbb{X})\to L^p(\mathbb{X})$ is a well-defined non-expansive operator for every~$p\in[1, \infty]$; see~\cite{BaRo95,BoGo12}.

The Perron--Frobenius operator~$\P$ can be seen as a linear, infinite-dimensional representation of the nonlinear, finite-dimensional dynamical system~$\Phi$. To see the connection, consider for some~$x\in\mathbb{X}$ the \emph{Dirac distribution}~$\delta_x(\cdot)$ as an element of~$L^1(\mathbb{X})$, with~$\int_{\mathbb{A}}\delta_x(y)\,d\mu(y)=1$ if~$x\in\mathbb{A}$ and~$0$ otherwise. Then
\begin{equation*}
    \int_{\mathbb{A}} \P \delta_x\,d\mu = \int_{\Phi^{-1}(\mathbb{A})} \delta_x\,d\mu = \int_{\mathbb{A}} \delta_{\Phi(x)}\,d\mu\,,
\end{equation*}
such that the Perron--Frobenius operator moves the center of the Dirac distribution in accordance with the dynamics.

\paragraph{Non-deterministic systems.}

We define the non-deterministic dynamical system~$\bm \Phi$ as a mapping acting on~$\mathbb{X}$ such that $\bm \Phi(x)$ is an~$\mathbb{X}$-valued random variable over some implicitly given probability space. We assume that~$\bm\Phi$ possesses a \emph{transition density function}~$k:\mathbb{X}\times\mathbb{X}\to\R_{\ge 0}$ satisfying
\begin{equation} \label{eq:tdf}
    \prob(\bm\Phi(x)\in \mathbb{A}) = \int_{\mathbb{A}} k(x,y)\,d\mu(y),\quad \mathbb{A}\in\mathfrak{B}\,.
\end{equation}
Here, $\prob$ denotes the probability with respect to the underlying probability space and~\eqref{eq:tdf} essentially means that~$ \bm\Phi(x)\sim k(x,\cdot) $. The existence of a transition density function can be seen as an analogue to non-singularity in the deterministic case: it ensures that~$\bm\Phi$ does not concentrate significant probability mass in sets of zero measure\footnote{Such a mapping is also called in the literature ``$\mu$-compatible'' or ``null preserving''~\cite{Hopf54,Kre85}.}.

For such systems, it can be quickly seen that the Perron--Frobenius operator satisfies
\begin{equation} \label{eq:FPOnondet}
    \P f(y) = \int f(x) k(x,y)\,d\mu(x)\,,
\end{equation}
and that the Markov operator property holds as well. If the measure~$\mu$ is invariant, i.e., $\mu(\mathbb{A}) = \int_{\mathbb{A}}\int k(x,y)\,d\mu(x)\,d\mu(y)$ for every~$\mathbb{A}\in\mathfrak{B}$, then~$\P: L^p(\mathbb{X})\to L^p(\mathbb{X})$ is a well-defined non-expansive operator for every~$p\in[1,\infty]$, as in the deterministic case.

\emph{Invariant} (or \emph{stationary}) \emph{densities} play a special role. These are densities~$f$ (i.e., positive functions with unit~$L^1$ norm) which satisfy~$\P f = f$. If such a density~$f$ is unique, the system is called \emph{ergodic}, and satisfies for any $g\in L^p(\mathbb{X})$, $p\in[1,\infty]$, that
\begin{equation} \label{eq:ergthm}
    \lim_{n\to\infty}\frac1n \sum_{k=0}^{n-1}g(\bm\Phi^k x) = \int g f\,d\mu
\end{equation}
$\prob$-almost surely (a.s.) for $\mu$-a.e.~$x\in\mathrm{supp}(f)$, where~$\mathrm{supp}(f)$ is the set~$\{f>0\}$. With some additional assumptions on~$k$, the convergence in~\eqref{eq:ergthm} is geometric, with the rate governed by the second dominant eigenvalue of~$\P$.
In general, eigenfunctions associated with subdominant eigenvalues correspond to the slowly converging transients of the system and yield information about \emph{metastable} sets; sets between which a dynamical transition is a rare event. For more details, we refer to~\cite{MeTw12,SS13}.

\subsection{Koopman operator}

While the Perron--Frobenius operator describes the evolution of \emph{densities}, the Koopman operator describes the evolution of \emph{observables}~\cite{BMM12}. An observable could, for instance, be a measurement or sensor probe. That is, instead of analyzing an orbit $ \{x, \, \Phi(x), \, \Phi^2(x), \,  \dots \} $ of the dynamical system, we now consider the measurements $ \{f(x), \, f(\Phi(x)), \, f(\Phi^2(x)), \, \dots \} $.

The Koopman operator $ \K:L^{\infty}(\mathbb{X})\to L^{\infty}(\mathbb{X}) $, see e.g.~\cite{BMM12, WKR14, FGH14a}, is defined by
\begin{equation}
    \K f = f \circ \Phi\,.
    \label{eq:Koopmanop}
\end{equation}
The Koopman operator~$ \K $ is the adjoint of the Perron--Frobenius operator $ \mathcal{P} $, i.e.
\begin{equation*}
    \innerprod{\mathcal{P} f}{g}_{\mu} = \innerprod{f}{\K g}_{\mu},
\end{equation*}
where~$ \innerprod{\cdot}{\cdot}_{\mu} $ is the duality pairing between~$L^1$ and~$L^{\infty}$ functions. For specific combinations of~$ \Phi $ and~$ \mu $, the Koopman operator can be defined on~$L^2(\mathbb{X})$, too\footnote{For instance, if the measure~$\mu$ is invariant under~$\Phi$~\cite{BaRo95}; or if~$k\in L^{\infty}(\mathbb{X}\times\mathbb{X})$.}; in what follows, we assume that this is the case.

Again, $ \K $ is an infinite-dimensional linear operator that characterizes the finite-dimensional nonlinear system~$ \Phi $. To obtain the dynamics of a system defined on $ \mathbb{X} \subset \R^d $, use the set of observables $ g_i(x) = x_i $, $ i = 1, \dots, d $, or in shorthand, the \emph{vector-valued} observable~$ g(x) = x $, where $ g $ is called \emph{full-state observable}. On vector-valued functions, the Koopman operator acts componentwise.

In order to maintain duality with the Perron--Frobenius operator, for the non-deterministic system~$\bm\Phi$ with transition density function~$ k $, the Koopman operator is defined as
\begin{equation*}
    \K f(x) = \expect\left[f(\bm\Phi(x))\right] = \int k(x,y)f(y)\,d\mu(y),
\end{equation*}
where~$\expect[\cdot]$ denotes the expectation value with respect to the probability measure underlying~$\bm\Phi(x)$. Note that while the Koopman operator was defined here for a discrete-time dynamical system, the definition can be extended naturally to continuous-time dynamical systems as described in~\cite{BMM12}.

If $ \varphi_1 $ and $ \varphi_2 $ are eigenfunctions of the Koopman operator with eigenvalues $ \lambda_1 $ and $ \lambda_2 $, then also the product $ \varphi_1 \, \varphi_2 $ is an eigenfunction with eigenvalue $ \lambda_1 \lambda_2 $. The product of two functions is defined pointwise, i.e.~$ (\varphi_1 \, \varphi_2)(x) = \varphi_1(x) \, \varphi_2(x) $. Analogously, for any eigenfunction $ \varphi $ and $ r \in \R $, $ \varphi^r $ is an eigenfunction with eigenvalue $ \lambda^r $ assuming that $ \varphi(x) \ne 0 $ for $ r < 0 $.

\begin{example} \label{ex:Linear example}
Consider a linear dynamical system of the form $ x_{k+1} = A \, x_k $ with $ A \in \R^{d \times d} $, cf.~\cite{BMM12, WKR14}. Let $ A $ have $ d $ left eigenvectors\footnote{Here and in what follows, left eigenvectors are represented as row vectors.} $ w_i $ with eigenvalues $ \mu_i $, i.e. $ w_i \, A = \mu_i \, w_i $ for $ i = 1, \dots, d $. Then $ \varphi_i(x) = w_i \, x $ is an eigenfunction of the Koopman operator $ \K $ with corresponding eigenvalue $ \lambda_i = \mu_i $ since
\begin{equation*}
    (\K \varphi_i)(x) = \varphi_i(A \, x)
                      = w_i \, A \, x
                      = \mu_i \, w_i \, x
                      = \mu_i \,\varphi_i(x).
\end{equation*}
As described above, also products of these eigenfunctions
\begin{equation*}
    \varphi_l(x) = \prod_{i=1}^d (w_i \, x)^{l_i}
\end{equation*}
are eigenfunctions with corresponding eigenvalue $ \lambda_l = \prod_{i=1}^d \lambda_i^{l_i} $, where $ l \in \mathbb{N}_0^d $ is a multi-index. For
\begin{equation*}
    A =
    \begin{bmatrix}[rr]
         0.48 & -0.06 \\
        -0.16 &  0.52
    \end{bmatrix},
\end{equation*}
for example, the left eigenvectors are $ w_1 = [ 0.8, \, -0.6 ] $ and $ w_2 = \tfrac{1}{\sqrt{5}} [ 2, \, 1 ] $ with eigenvalues $ \mu_1 = 0.6 $ and $ \mu_2 = 0.4 $. The first eight nontrivial eigenfunctions of the Koopman operator are shown in Figure~\ref{fig:Linear example}. \exampleSymbol

\begin{figure}[htb]
    \centering
    \includegraphics[width=0.9\textwidth]{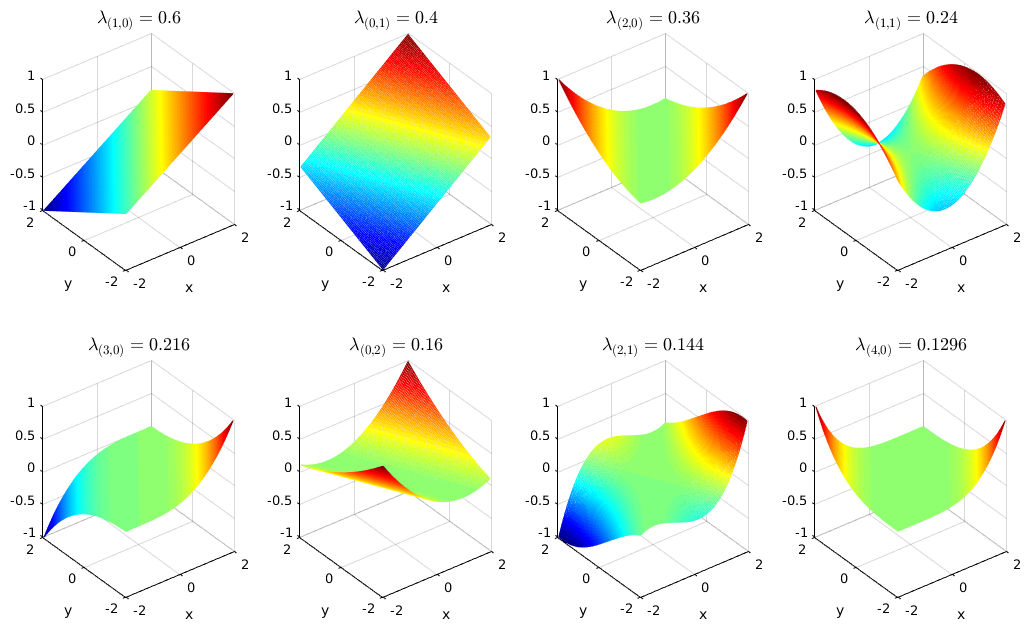}
    \caption{Eigenfunctions of the Koopman operator for the linear dynamical system described in Example~\ref{ex:Linear  example}.}
    \label{fig:Linear example}
\end{figure}
\end{example}

Let $ f : \mathbb{X} \to \R $ be an observable of the system that can be written as a linear combination of the linearly independent eigenfunctions $ \varphi_i $, i.e.
\begin{equation*}
    f(x) = \sum_i c_i \, \varphi_i(x),
\end{equation*}
with $ c_i \in \C $. Then
\begin{equation*}
    (\K f)(x) = \sum_i \lambda_i \, c_i \, \varphi_i(x).
\end{equation*}
Analogously, for vector-valued functions $ F = [f_1, \, \dots, \, f_n]^T $, we get
\begin{equation*}
    \K F =
    \begin{bmatrix}
        \sum_i \lambda_i \, c_{i, 1} \, \varphi_i \\
        \vdots \\
        \sum_i \lambda_i \, c_{i, n} \, \varphi_i
    \end{bmatrix} =
    \sum_i \lambda_i \, \varphi_i
    \begin{bmatrix}
        c_{i, 1} \\
        \vdots \\
        c_{i, n}
    \end{bmatrix} =
    \sum_i \lambda_i \, \varphi_i \, v_i,
\end{equation*}
where $ v_i = [c_{i, 1}, \, \dots, \, c_{i, n}]^T $. These vectors $ v_i $ corresponding to the eigenfunctions $ \varphi_i $ are called Koopman modes.

\begin{definition}
Given an eigenfunction $ \varphi_i $ of the Koopman operator $ \K $ and a vector-valued observable $ F $, the vector $ v_i $ of coefficients of the projection of $ F $ onto $ \mspan\{ \varphi_i \} $ is called Koopman mode.
\end{definition}

The connection between the dynamical system $ \Phi $ and the Koopman operator $ \K $ is given by the full-state observable $ g(x) = x $ and the corresponding Koopman eigenvalues $ \lambda_i $, eigenfunctions $ \varphi_i $, and eigenmodes $ v_i $ required to retrieve the full state~\cite{WKR14}. Since $ (\K g)(x) = (g \circ \Phi)(x) = \Phi(x) $ and, using the Koopman modes $ v_i $ belonging to $ g $,
\begin{equation*}
    (\K g)(x) = \sum_i \lambda_i \, \varphi_i(x) \, v_i,
\end{equation*}
we can compute $ \Phi(x) $ with the aid of the Koopman operator. A pictorial representation of the relationship between states and observables as well as the evolution operator and Koopman operator can be found in~\cite{WKR14}.

\section{Numerical approximation}
\label{sec:Numerical approximation}

\subsection{Generalized Galerkin methods}

The Galerkin discretization of an operator $ \mathcal{A} $ over some Hilbert space $\mathbb H$ can be described as follows. Suppose we have a finite-dimensional subspace $\mathbb V \subset \mathbb H$ with basis $(\psi_1,\ldots,\psi_k)$ given. The \emph{Galerkin projection} of $\mathcal A$ to $\mathbb V$ is the unique linear operator $ A: \mathbb{V} \to \mathbb{V} $ satisfying
\begin{equation} \label{eq:Galerkin}
    \innerprod{\psi_j}{\mathcal A\psi_i} = \innerprod{\psi_j}{A\psi_i},
    \quad \text{for all } i,j = 1,\ldots,k\,.
\end{equation}
If the operator $ \mathcal{A} $ is not given on a Hilbert space, just a Banach space, it can be advantageous to take \emph{basis functions} (with respect to which the projected operator is defined) and \emph{test functions} (which serve in \eqref{eq:Galerkin} to project objects not necessarily living in the same subspace) from different sets.

If $\mathcal A: \mathbb Y \to \mathbb Y$ is an operator on a Banach space $\mathbb Y$, $\mathbb V\subset \mathbb Y$ a subspace with basis $(\psi_1,\ldots,\psi_k)$, $\mathbb W \subset\mathbb Y^*$ a subspace of the dual of~$\mathbb Y$ with basis~$(\psi_1^*,\ldots,\psi_k^*)$, i.e.\ $\dim\mathbb V = \dim\mathbb W$, then the \emph{Petrov--Galerkin} projection of~$\mathcal A$ is the unique linear operator~$ A:\mathbb V \to \mathbb{V} $ satisfying
\begin{equation} \label{eq:PetrovGalerkin}
    \innerprod{\psi_j^*}{\mathcal A\psi_i} = \innerprod{\psi_j^*}{A\psi_i},
    \quad \text{for all } i,j = 1,\ldots,k\,,
\end{equation}
where $ \innerprod{\cdot}{\cdot} $ denotes the duality bracket.

This idea can be taken one step further, resulting in a Petrov--Galerkin-like projection even if $ l:=\dim \mathbb W > \dim \mathbb V $. In this case, \eqref{eq:PetrovGalerkin} is over-determined and the projected operator~$ A $ is defined as the solution of the least-squares problem
\begin{equation} \label{eq:PetrovGalerkinLS}
    \sum_{j=1}^l\sum_{i=1}^k \innerprod{\psi_j^*}{\mathcal A\psi_i-A \psi_i}^2 = \min!
\end{equation}
We refer to this as the \emph{over-determined Petrov--Galerkin} method.

\subsection{Ulam's method}

Probably the most popular method to date for the discretization of the Perron--Frobenius operator is Ulam's method; see e.g.~\cite{Ulam60,CU02, BS13, FGH14a}. Let $ \{ \mathbb{B}_1, \, \dots, \, \mathbb{B}_k \} \subset \mathfrak{B} $ be a covering of~$ \mathbb{X} $ by a finite number of disjoint measurable boxes and let~$ \mathds{1}_{\mathbb{B}_i} $ be the indicator function for box~$ \mathbb{B}_i $, i.e.
\begin{equation*}
    \mathds{1}_{\mathbb{B}_i}(x)
        = \begin{cases} 1,& \text{if } x \in \mathbb{B}_i, \\ 0,& \text{otherwise}. \end{cases}
\end{equation*}
Ulam's method is a Galerkin projection of the Perron--Frobenius operator to the subspace spanned by these indicator functions. More precisely, if one chooses the basis functions~$\psi_i = \tfrac{1}{\mu(\mathbb{B}_i)}\mathds{1}_{\mathbb{B}_i}$, then from the relationship
\begin{equation} \label{eq:Ulam_denum}
    \begin{split}
        \int \mathds{1}_{\mathbb{B}_j} \cdot \mathcal{P} \mathds{1}_{\mathbb{B}_i} \, d\mu
            &= \int (\mathds{1}_{\mathbb{B}_j} \circ \Phi) \cdot \mathds{1}_{\mathbb{B}_i} \, d\mu
            = \int \mathds{1}_{\Phi^{-1}(\mathbb{B}_j)} \cdot \mathds{1}_{\mathbb{B}_i} \, d\mu \\[2mm]
            &= \mu(\Phi^{-1}(\mathbb{B}_j) \cap \mathbb{B}_i)
    \end{split}
\end{equation}
we can represent the discrete operator by a matrix~$ P = (p_{ij}) \in \R^{k \times k} $ with
\begin{equation} \label{eq:p_ij}
    p_{ij} = \frac{\mu\left(\Phi^{-1}(\mathbb{B}_j) \cap \mathbb{B}_i\right)}{\mu(\mathbb{B}_i)}.
\end{equation}
The denominator $ \mu(\mathbb{B}_i) $ normalizes the entries~$ p_{ij} $ so that~$ P $ becomes a row-stochastic matrix. Thus, $ P $ defines a finite Markov chain and has a left eigenvector with the corresponding eigenvalue $ \lambda_1 = 1 $. This eigenvector approximates the invariant measure of the Perron--Frobenius operator~$ \mathcal{P} $~\cite{Li76, MurrPhD, Fr98, DJ99}.

The entries~$ p_{ij} $ of the matrix~$ P $ can be viewed as the probabilities of being mapped from box~$ \mathbb{B}_i $ to box~$ \mathbb{B}_j $ by the dynamical system~$\Phi$. These entries can be estimated by randomly choosing a large number of test points $ x_i^{(l)} $, $ l = 1, \dots, n $, in each box $ \mathbb{B}_i $ and counting the number of times $ \Phi(x_i^{(l)}) $ is contained in box $ \mathbb{B}_j $, that is,
\begin{equation} \label{eq:p_ij_approx}
    p_{ij} \approx \frac{1}{n} \sum_{l=1}^{n} \mathds{1}_{\mathbb{B}_j}(\Phi(x_i^{(l)}))\,.
\end{equation}
On the one hand, this is a Monte-Carlo approach to estimate the integrals in~\eqref{eq:Ulam_denum}, and hence a \emph{numerical realization} of Ulam's method. On the other hand, it is also an over-determined Petrov--Galerkin method~\eqref{eq:PetrovGalerkinLS} with test functionals~$\psi_{\ell}^*$ being point evaluations at the respective sample points~$x_{\ell}$; i.e., for a piecewise continuous function~$\varphi$ we have $\psi_{\ell}^*(\varphi) = \int\varphi\delta_{x_{\ell}}\,d\mu = \varphi(x_{\ell})$. One can see this by noting that due to the disjoint support of the basis functions~$\mathds{1}_{\mathbb{B}_i}$ the sum in~\eqref{eq:PetrovGalerkinLS} decouples and the entries of~$ P $ can be readily seen to be as on the right-hand side of~\eqref{eq:p_ij_approx}. The effect of Monte-Carlo sampling and the choice of the partition on the accuracy and convergence of Ulam's method has been investigated in~\cite{BoMu01, Mur04, KPPHD}.

\begin{remark}
We note that, given independent random test points $ x_i^{(l)}\in\mathbb{B}_i $, $ l = 1, \ldots, n $, expression~\eqref{eq:p_ij_approx} is a maximum-likelihood estimator for \eqref{eq:p_ij}.
This holds true in the non-deterministic case as well, where~\eqref{eq:p_ij} reads as
\begin{equation*}
    p_{ij} = \prob\left(\bm{\Phi}(x_i^{(l)})\in \mathbb{B}_j\right), 
\end{equation*}
and the~$ \Phi(x_i^{(l)}) $ in~\eqref{eq:p_ij_approx} are replaced by mutually independent realizations of~$ \bm\Phi(x_i^{(l)}) $, $ l = 1,\dots,n $.
\end{remark}

\subsection{Further discretization methods for the Perron--Frobenius operator}

\paragraph{Petrov--Galerkin type and higher order methods.}

Ulam's method is a zeroth order method in the sense that it uses piecewise constant basis functions. We can achieve a better approximation of the operator (and its dominant spectrum, in particular) if we use higher order piecewise polynomials in the Galerkin approximation; see~\cite{DiDuLi,DiZh96}.

If the eigenfunctions of the Perron--Frobenius operator are expected to have further regularity, the use of \emph{spectral methods} can be advantageous~\cite{Hub09,FrJuKo13}. Here, \emph{collocation} turns out to be the most efficient, in general; i.e., where basis functions are Fourier or Chebyshev polynomials~\cite{Boyd01}, and test functions are Dirac distributions centered in specific domain-dependent collocation points. \emph{Mesh-free} approaches with radial basis functions continuously gain popularity due to their flexibility  with respect to state space geometry~\cite{FrJu15,WRR15}.

A kind of regularity different from smoothness is if functions of interest do not vary simultaneously strongly in many coordinates, just in very few of them. \emph{Sparse-grid} type Galerkin approximation schemes~\cite{BuGr04} are well suited for such objects; their combination with Ulam's method has been considered in~\cite{JuKo09}.

Higher-order approximations do have, however, an unwanted disadvantage: the discretized operator is not a Markov operator (a stochastic matrix), in general~\cite[Section 3]{KPPHD}. This desirable structural property can be retained if one considers specific Petrov--Galerkin methods; cf.~\cite{DiLi91}, where the basis functions are piecewise first- or second-order polynomials and the test functions are piecewise constant.

\paragraph{Maximum entropy optimization methods.}

Let us consider a Petrov--Galerkin method for discretizing the Perron--Frobenius operator~$\P$, such that $\psi^*_j(\varphi) = \int h_j\varphi\,d\mu$ for suitable~$h_j\in L^{\infty}(\mathbb{X})$, $j=1,\ldots,k$. Then the image~$Pf$ of~$f\in\mathbb{V}$ is the unique element~$g\in\mathbb{V}$ such that
\begin{equation} \label{eq:momenteq1}
    \int (\P f - g)h_j\,d\mu = \int f ( h_j\circ\Phi) - g h_j\,d\mu = 0, \quad j=1,\ldots,k\,.
\end{equation}
One might as well alleviate the condition~$g\in\mathbb{V}$, at the cost of not having a unique solution to~\eqref{eq:momenteq1}. Then, in order to get a unique solution, one has to impose additional conditions on~$ g $. If one considers~\eqref{eq:momenteq1} as \emph{constraints}, one could formulate an optimization problem whose solution is $ g $. There is, of course, no trivial choice of objective functional for this optimization problem, however \emph{energy-type} (i.e.\ $\int g^2\,d\mu$) and \emph{entropy-type} (i.e.\ $\int g\log g\,d\mu$) objective functionals turned out to be advantageous to use~\cite{Ding98,BoMu06a,BoMu06b,BoMu07}. The reason for this is that the available convergence analysis for Ulam's method is quite restrictive~\cite{Li76,DiZh96,Fr98}, and these optimization-based methods yield novel convergent schemes to approximate invariant densities of non-singular dynamical systems -- to this end, one sets~$g=f$ in~\eqref{eq:momenteq1}. The down-side of this method is that in order to represent the approximate invariant density, one has to compute ``basis functions'' which arise as non-trivial combinations of the test functions~$h_j$ and the dynamics~$\Phi$.

\subsection{Extended dynamic mode decomposition}
\label{ssec:EDMD}

An approximation of the Koopman operator, the Koopman eigenvalues, eigenfunctions, and eigenmodes can be computed using \emph{Extended Dynamic Mode Decomposition} (EDMD). Note that we are using a slightly different notation than~\cite{WKR14, WRK14} here to make the relationship with other methods, in particular Ulam's method and \emph{Dynamic Mode Decomposition} (DMD, defined in Remark~\ref{rem:DMD} below), more apparent. In order to obtain EDMD, we take the basis functions $ \psi_i $, as above, and for the test function(al)s, we take delta distributions $ \delta_{x_j} $, that is, $ \innerprod{\delta_x}{\psi} = \psi(x)$. EDMD requires data, i.e.\ a set of values $ x_i $ and the corresponding $ y_i = \Phi(x_i) $ values, $ i = 1, \dots, m $, written in matrix form
\begin{equation} \label{eq:X and Y}
    X =
    \begin{bmatrix}
        x_1 & \cdots & x_m
    \end{bmatrix}
    \quad \text{and} \quad
    Y =
    \begin{bmatrix}
        y_1 & \cdots & y_m
    \end{bmatrix},
\end{equation}
and additionally a set of basis functions or observables
\begin{equation*}
    \mathbb{D} = \left\{ \psi_1, \, \psi_2, \, \dots, \, \psi_k \right\}
\end{equation*}
called dictionary. EDMD takes ideas from collocation methods, which are, for example, used to solve PDEs, where the $ x_i $ are the collocation points rather than a fixed grid~\cite{WRK14}. Writing
\begin{equation*}
    \Psi =
    \begin{bmatrix}
        \psi_1 & \psi_2 & \cdots & \psi_k
    \end{bmatrix}^T
\end{equation*}
as a vector of functions, that is $ \Psi : \mathbb{X} \to \R^k $, this yields
\begin{equation*}
    \Psi_Y^T = \Psi_X^T K,
\end{equation*}
with
\begin{equation*}
    \Psi_X =
    \begin{bmatrix}
        \Psi(x_1) & \dots & \Psi(x_m)
    \end{bmatrix}
    \quad \text{and} \quad
    \Psi_Y =
    \begin{bmatrix}
        \Psi(y_1) & \dots & \Psi(y_m)
    \end{bmatrix},
\end{equation*}
i.e.~$ \Psi_X, \Psi_Y \in \R^{k \times m} $. Here, $ K \in \mathbb{R}^{k \times k} $ applied from the right to vectors in $ \R^{1 \times k} $ represents the projection of $ \K $ with respect to the basis $ (\psi_1, \ldots, \psi_k) $. If the number of basis functions and test functions does not match, \eqref{eq:PetrovGalerkin} cannot be satisfied in general and a least squares solution of the (usually overdetermined) system of equations is given by applying~$ \Psi_X^+ $, the pseudoinverse of $ \Psi_X $, giving
\begin{equation} \label{eq:EDMD simple}
    K^T = \Psi_Y \, \Psi_X^+.
\end{equation}
A more detailed description can be found in Appendix~\ref{sec:Derivation of EDMD}. For the sake of convenience and to compare DMD and EDMD, we define $ M_\textsub{K} = K^T $. This approach becomes computationally expensive for large $ m $ since it requires the  pseudoinverse of the $ k \times m $ matrix $ \Psi_X $. Another possibility to compute $ K $ is
\begin{equation*}
    K^T = A \, G^+,
\end{equation*}
where the matrices $ A, \, G \in \R^{k \times k} $ are given by
\begin{equation} \label{eq:A and G entries}
    \begin{split}
        A &= \frac{1}{m} \sum_{l=1}^m \Psi(y_l) \, \Psi(x_l)^T, \\
        G &= \frac{1}{m} \sum_{l=1}^m \Psi(x_l) \, \Psi(x_l)^T.
    \end{split}
\end{equation}
In order to obtain the second EDMD formulation from the first, the relationship $ \Psi_X^+ = \Psi_X^T \, (\Psi_X \, \Psi_X^T)^+ $ was used. For a detailed derivation of these results, we refer to~\cite{WKR14, WRK14}.

An approximation of the eigenfunction $ \varphi_i $ of the Koopman operator $ \K $ is then given by
\begin{equation*}
    \varphi_i = \xi_i \, \Psi,
\end{equation*}
where $ \xi_i $ is the $ i $-th left eigenvector of the matrix $ M_\textsub{K} = K^T $.

\begin{example}
Let us consider the linear system described in Example~\ref{ex:Linear example} again. The eigenfunctions computed using EDMD with the basis functions $ \psi_l = x_1^{l_1} \, x_2^{l_2} $, $ 0 \le l_1, l_2 \le 5 $, are in very good agreement with the theoretical results. EDMD computes exactly the eigenfunctions shown in Figure~\ref{fig:Linear example} with negligibly small numerical errors $ \varepsilon < 10^{-10} $, where we computed the maximum difference between the eigenfunctions and their approximation. The first eight nontrivial eigenvalues of $ M_\textsub{K} $ are
\begin{align*}
    \lambda_2 &= 0.6000, &
    \lambda_3 &= 0.4000, &
    \lambda_4 &= 0.3600, &
    \lambda_5 &= 0.2400, \\
    \lambda_6 &= 0.2160, &
    \lambda_7 &= 0.1600, &
    \lambda_8 &= 0.1440, &
    \lambda_9 &= 0.1296. \tag*{\exampleSymbol}
\end{align*}
\end{example}

In order to obtain the Koopman modes for the full-state observable $ g(x) = x $  introduced above, define $ \boldsymbol{\varphi} = \left[ \varphi_1, \, \dots, \, \varphi_k \right]^T $ and let $ B \in \R^{d \times k} $ be the matrix such that $ g = B \, \Psi $, then $ \boldsymbol{\varphi} = \Xi \, \Psi $ and
\begin{equation*}
    g = B \, \Psi = B \, \Xi^{-1} \boldsymbol{\varphi},
\end{equation*}
where the matrix
\begin{equation*}
    \Xi =
    \begin{pmatrix}
        \xi_1 \\
        \xi_2 \\
        \vdots \\
        \xi_k
    \end{pmatrix}
\end{equation*}
contains all left eigenvectors of $ M_\textsub{K} $. Thus, the column vectors of the $ (d \times k) $-dimensional matrix $ V = B \, \Xi^{-1} $ are the Koopman modes $ v_i $ and
\begin{equation*}
    g = \sum_i \varphi_i \, v_i
    \quad \Rightarrow \quad
    \K g = \sum_i \lambda_i \, \varphi_i \, v_i.
\end{equation*}
Note that since $ \Xi $ is the matrix which contains all left eigenvectors of $ M_\textsub{K} $, the matrix $ \Xi^{-1} $ needed for reconstructing the full-state observable $ g $ contains all right eigenvectors of $ M_\textsub{K} $. That is, the Koopman eigenfunctions $ \boldsymbol{\varphi} = \Xi \, \Psi $ are approximated by the left eigenvectors of $ M_\textsub{K} $ and the Koopman modes $ V = B \, \Xi^{-1} $ by the right eigenvectors (cf.~\cite{WKR14}, with the difference that there the observables and eigenfunctions are written as column vectors and the data matrices $ \Psi_X $ and $ \Psi_Y $ are the transpose of our matrices; we chose to rewrite the EDMD formulation in order to illustrate the similarities with DMD and other methods).

\begin{example}
For Example~\ref{ex:Linear example} and the dictionary $ \psi_l = x_1^{l_1} \, x_2^{l_2} $, $ 0 \le l_1, l_2 \le 5 $, the matrix $ B \in \R^{2 \times 36} $ is zero except for the two entries corresponding to the functions $ \psi_{(1, 0)} = x_1 $ and $ \psi_{(0, 1)} = x_2 $. Thus, only the two eigenmodes $ v_{(1, 0)} \approx [0.5, \, -1]^T $ and $ v_{(0, 1)} \approx \frac{\sqrt{5}}{2} [0.6, \, 0.8 ]^T $ -- eigenvectors of $ A $ -- are required to construct the full-state observable, all the other eigenmodes are numerically zero. \exampleSymbol
\end{example}

\begin{remark}[\textbf{Convergence of EDMD to a Galerkin method}]
As described in~\cite{WKR14}, EDMD converges to a Galerkin approximation of the Koopman operator for large $ m $ if the data points are drawn according to a distribution~$ \mu $. Using the Galerkin approach, we would obtain matrices $ \widetilde{A} $ and $ \widetilde{G} $ with entries
\begin{equation*}
    \begin{split}
        \widetilde{a}_{ij} &= \innerprod{\K \psi_i}{\psi_j}_\mu, \\
        \widetilde{g}_{ij} &= \innerprod{\psi_i}{\psi_j}_\mu.
    \end{split}
\end{equation*}
Here, $ \innerprod{f}{g}_\mu = \int_\mathbb{X} f(x) \, g^*(x) \, d\mu(x) $. Then $ \widetilde{K}^T = \widetilde{A} \, \widetilde{G}^{-1} $ would be the finite-dimensional approximation of the Koopman operator $ \K $. Clearly, the entries $ a_{ij} $ and $ g_{ij} $ of the matrices $ A $ and $ G $ in \eqref{eq:A and G entries} converge to $ \widetilde{a}_{ij} $ and $ \widetilde{g}_{ij} $ for $ m \rightarrow \infty $, since
\begin{equation*}
    \begin{split}
        a_{ij} &= \frac{1}{m} \sum_{l=1}^m \psi_i(y_l) \psi_j(x_l)^*
                \underset{\scriptscriptstyle m \rightarrow \infty}{\longrightarrow} \int_\mathbb{X} (\K \psi_i)(x) \psi_j(x)^* d\mu(x)
                = \innerprod{\K \psi_i}{\psi_j}_\mu = \widetilde{a}_{ij}, 
    \end{split}
\end{equation*}  \begin{equation} \label{eq:EDMDtoGalerkin}
    \begin{split}      g_{ij} &= \frac{1}{m} \sum_{l=1}^m \psi_i(x_l) \psi_j(x_l)^*
                \underset{\scriptscriptstyle m \rightarrow \infty}{\longrightarrow} \int_\mathbb{X} \psi_i(x) \psi_j(x)^* d\mu(x)
                = \innerprod{\psi_i}{\psi_j}_\mu = \widetilde{g}_{ij}.
    \end{split}
\end{equation}
\end{remark}

\begin{remark}[\textbf{Variational approach for reversible processes}]
The EDMD approximation of the eigenfunctions of the Koopman operator is given by the left eigenvectors $ \xi $ of the matrix $ M_\textsub{K} = A \, G^+ $, i.e.~$ \xi \, M_\textsub{K} = \lambda \, \xi $, and can be -- provided that $ G $ is regular -- reformulated as a generalized eigenvalue problem of the form $ \xi \, A = \lambda \, \xi \,G $. This results in a method similar to the variational approach presented in~\cite{NoNu13} for reversible processes. A tensor-based generalization of this method can be found in \cite{NSVN15}.
\end{remark}

\begin{remark}[\textbf{DMD}] \label{rem:DMD}
Dynamic Mode Decomposition was first introduced in~\cite{SS08} and is a powerful tool for analyzing the behavior of nonlinear systems which can, for instance, be used to identify low-order dynamics of a system~\cite{TRLBK13}. DMD analyzes pairs of $ d $-dimensional data vectors $ x_i $ and $ y_i = \Phi(x_i) $, $ i = 1, \dots, m $, written again in matrix form \eqref{eq:X and Y}. Assuming there exists a linear operator $ M_\textsub{L} $ that describes the dynamics of the system such that $ y_i = M_\textsub{L} \, x_i $, define $ M_\textsub{L} = Y X^+ $. The DMD modes and eigenvalues are then defined to be the eigenvectors and eigenvalues of $ M_\textsub{L} $. The matrix $ M_\textsub{L} $ minimizes the cost function $ \norm{M_\textsub{L} \, X - Y}_F $, where $ \norm{.}_F $ is the Frobenius norm. There are different algorithms to compute the DMD modes and eigenvalues without explicitly computing $ M_\textsub{L} $ which rely on the (reduced) singular value decomposition of $ X $. For a detailed description, we refer to~\cite{TRLBK13}.
\end{remark}

\begin{remark}[\textbf{DMD and EDMD}]
The first EDMD formulation~\eqref{eq:EDMD simple} shows the relationship between DMD and EDMD. Let the vector of observables be given by $ \Psi(x) = x $. Then $ \Psi_X = X $ and $ \Psi_Y = Y $, thus
\begin{equation*}
    M_\textsub{K} = \Psi_Y \, \Psi_X^+ = Y \, X^+ = M_\textsub{L},
\end{equation*}
i.e.~the DMD matrix $ M_\textsub{L} $ is an approximation of the Koopman operator $ \mathcal{K} $ using only linear basis functions. Since $ B = I $, the Koopman modes are $ V = \Xi^{-1} $, which are the right eigenvectors of $ M_\textsub{K} $ and thus the right eigenvectors of $ M_\textsub{L} $, which illustrates that the Koopman modes in this case are the DMD modes. Hence, (exact) DMD can be regarded as a special case of EDMD.
\end{remark}

\begin{remark}[\textbf{Sparsity-promoting DMD}]
A variant of DMD aiming at maximizing the quality of the approximation while minimizing the number of modes used to describe the data is presented in \cite{JSN14}. Sparsity is achieved by using an $ \ell_1 $-norm regularization approach. The $ \ell_1 $-norm can be regarded as a convexification of the cardinality function. The resulting regularized convex optimization problem is then solved with an alternating direction method. That is, the algorithm alternates between minimizing the cost function and maximizing sparsity.
\end{remark}

In the same way, a sparsity-promoting version of EDMD could be constructed in order to minimize the number of basis functions required for the representation of the eigenfunctions.

\subsection{Kernel-based extended dynamic mode decomposition}

In some cases, it is possible to improve the efficiency of EDMD using the so-called \emph{kernel trick}~\cite{WRK14}. In fluid problems, for example, the number of measurement points $ k $ is typically much larger than the number of measurements or snapshots $ m $. Suppose $ f(x, y) = (1 + x^T \, y)^2 $ for $ x, y \in \R^2 $, then
\begin{equation*}
    \begin{split}
        f(x, y) = 1 + 2 \, x_1 \, y_1 + 2 \, x_2 \, y_2 + 2 \, x_1 \, x_2 \, y_1 \, y_2
                + x_1^2 \, y_1^2 + x_2^2 \, y_2^2
                = \Psi(x)^T \, \Psi(y)
    \end{split}
\end{equation*}
for the vector of observables $ \Psi(x) = \left[1, \,\sqrt{2} x_1, \,\sqrt{2} x_2, \,\sqrt{2} x_1 \, x_2, \, x_1^2, \,x_2^2\right]^T $. The kernel function $ f(x, y) = (1 + x^T \, y)^p $ for $ x, y \in \R^d $ will generate a vector-valued observable that contains all monomials of order up to and including $ p $. That is, instead of $ \mathcal{O}(k) $, the computation of the inner product is now $ \mathcal{O}(d) $ since inner products are computed implicitly by an appropriately chosen kernel function.

In~\cite{WRK14}, it is shown that any left eigenvector $ v $ of $ M_\textsub{K} $ for an eigenvalue $ \lambda \ne 0 $ can be written as $ v = \hat{v} \, \Psi_X^T $, with $ \hat{v} \in \R^m $. Using the relationship $ \Psi_X^+ = (\Psi_X^T \, \Psi_X)^+ \Psi_X^T $, we then obtain
\begin{equation*}
    \setlength\arraycolsep{1pt}
    \begin{array}{cl}
        v \, M_\textsub{K} &= \hat{v} \, \Psi_X^T \, M_\textsub{K} = \hat{v} \, \Psi_X^T (\Psi_Y \Psi_X^+)
                            = \hat{v} \, (\Psi_X^T \, \Psi_Y) (\Psi_X^T \, \Psi_X)^+ \Psi_X^T
                            = \hat{v} \, \hat{M}_\textsub{K} \, \Psi_X^T \\
        \verteq  & \\[-0.4em]
        \mu \, v &= \mu \, \hat{v} \, \Psi_X^T
    \end{array}
\end{equation*}
and thus a left eigenvector of $ M_\textsub{K} $ can be computed by a left eigenvector of $ \hat{M}_\textsub{K} = \hat{K}^T = \hat{A} \, \hat{G}^+ $ multiplied by $ \Psi_X^T $, where $ \hat{A} = \Psi_X^T \, \Psi_Y \in \R^{m \times m} $ and $ \hat{G} = \Psi_X^T \, \Psi_X \in \R^{m \times m} $. The entries of the matrices $ \hat{A} $ and $ \hat{G} $ can be computed efficiently by
\begin{equation*}
    \begin{split}
        \hat{a}_{ij} &= f(x_i, y_j), \\
        \hat{g}_{ij} &= f(x_i, x_j),
    \end{split}
\end{equation*}
using the kernel function $ f $. The computational cost for the eigenvector computation now depends on the number of snapshots $ m $ rather than the number of observables $ k $. For a more detailed description, we refer to~\cite{WRK14}.

\section{Duality}
\label{sec:Duality}

In this section, we will show how, given the eigenfunctions of the Koopman operator, the eigenfunctions of the adjoint Perron--Frobenius operator can be computed, or vice versa. The goal here is to illustrate the similarities between the different numerical methods presented in the previous sections and to adapt methods developed for one operator to compute eigenfunctions of the other operator. We will focus in particular on Ulam's method and EDMD.

\subsection{Ulam's method and EDMD}
Let us consider the case where the dictionary contains the indicator functions for a given box discretization $ \{ \mathbb{B}_1, \, \dots, \, \mathbb{B}_k \} $, i.e.\ $ \mathbb{D} = \{ \mathds{1}_{\mathbb{B}_1}, \, \dots, \, \mathds{1}_{\mathbb{B}_k} \} $. If we now select $ n $ test points $ x_i^{(l)} $, $ l = 1, \dots, n $, for each box, then
\begin{equation*}
    \Psi_X =
    \begin{pmatrix}
        \mathds{1}_n^T &                &        &                \\
                       & \mathds{1}_n^T &        &                \\
                       &                & \ddots &                \\
                       &                &        & \mathds{1}_n^T
    \end{pmatrix}
    \in \R^{k \times k \,n},
\end{equation*}
where $ \mathds{1}_n \in \R^n $ is the vector of all ones. The pseudoinverse of this matrix is $ \Psi_X^+ = \frac{1}{n} \Psi_X^T $ and the matrix $ M_\textsub{K} = \Psi_Y \, \Psi_X^+ \in \R^{k \times k} $ with entries $ \underbar{m}_{ij} $ has the following form
\begin{equation*}
    \underbar{m}_{ij} = \sum_{l=1}^{k\,n} (\Psi_Y)_{il} \, (\Psi_X)^+_{lj}
                      = \frac{1}{n} \sum_{l=1}^n \psi_i(y_j^{(l)}) \\
                      = \frac{1}{n} \sum_{l=1}^n \mathds{1}_{\mathbb{B}_i}(\Phi(x_j^{(l)})).
\end{equation*}
Comparing the entries $ \underbar{m}_{ij} $ of $ M_\textsub{K} $ with the entries $ p_{ij} $ of $ P $ in \eqref{eq:p_ij_approx}, it turns out that $ M_\textsub{K} = P^T $ and thus $ P = K $. That is, EDMD with indicator functions for a given box discretization computes the same finite-dimensional representation of the operators as Ulam's method.

\subsection{Computation of the dual basis}

For the finite-dimensional approximation, let $ \varphi_i $ be the eigenfunctions of $ \K $ and $ \widetilde{\varphi}_i $ the eigenfunctions of the adjoint operator $ \mathcal{P} $, $ i = 1, \dots, k $. Since
\begin{equation*}
        \innerprod{\K \varphi_i}{\widetilde{\varphi}_j}_\mu
            = \lambda_i \innerprod{\varphi_i}{\widetilde{\varphi}_j}_\mu \quad \text{and} \quad
        \innerprod{\varphi_i}{\mathcal{P} \widetilde{\varphi}_j}_\mu
            = \lambda_j \innerprod{\varphi_i}{\widetilde{\varphi}_j}_\mu,
\end{equation*}
subtracting these two equations gives $ 0 = (\lambda_i - \lambda_j) \innerprod{\varphi_i}{\widetilde{\varphi}_j}_\mu $. The left-hand side of the equation is zero due to the definition of the adjoint operator. Thus, if $ \lambda_i \ne \lambda_j $, the scalar product must be zero. Furthermore, $ \widetilde{\varphi}_j $ can be scaled in such a way that $ \innerprod{\varphi_i}{\widetilde{\varphi}_i}_\mu = 1 $. Hence, we can assume that $ \innerprod{\varphi_i}{\widetilde{\varphi}_j}_\mu = \delta_{ij} $.

Let now $ B = (b_{ij}) \in \C^{k \times k} $ and $ C = (c_{ij}) \in \C^{k \times k} $. Define $ b_{ij} = \innerprod{\varphi_i}{\varphi_j}_\mu $ and write
\begin{equation*}
    \widetilde{\varphi}_j = \sum_{l=1}^k c_{jl} \, \varphi_l,
\end{equation*}
then
\begin{equation*}
    \innerprod{\varphi_i}{\widetilde{\varphi}_j}_\mu
        = \innerprod{\varphi_i}{\sum_{l=1}^k c_{jl} \, \varphi_l}_\mu
        = \sum_{l=1}^k c_{jl}^* \innerprod{\varphi_i}{\varphi_l}_\mu
        = \sum_{l=1}^k b_{il} \, c_{jl}^*
        = \sum_{l=1}^k b_{il} \, c_{l j}.
\end{equation*}
It follows that the coefficients $ c_{ij} $ have to be chosen such that $ C = B^{-1} $. In order to obtain the matrix $ B $, we compute
\begin{equation*}
    \begin{split}
        b_{ij} &= \innerprod{\varphi_i}{\varphi_j}_\mu \approx \innerprod{\xi_i \, \Psi}{\xi_j \, \Psi}_\mu
               = \frac{1}{m} \sum_{l=1}^m \left(\xi_i \, \Psi(x_l) \right) \left(\xi_j \, \Psi(x_l) \right)^* \\
               &= \frac{1}{m} \left(\xi_i \, \Psi_X\right) \left(\xi_j \, \Psi_X \right)^*
               = \frac{1}{m} \, \xi_i \, G \, \xi_j^*,
    \end{split}
\end{equation*}
where again $ G = \Psi_X \Psi_X^T $. That is,
\begin{equation*}
    B = \frac{1}{m} \Xi \, G \, \Xi^* \quad \Rightarrow \quad
    C = B^{-1} = m \, (\Xi^*)^{-1} \, G^{-1} \, \Xi^{-1}.
\end{equation*}
Here, we assume that the matrix $ G $ is invertible. It follows that
\begin{equation*}
    \widetilde{\varphi}_j = \sum_{l=1}^k c_{jl} \, \xi_l \, \Psi = \widetilde{\xi}_j \, \Psi,
\end{equation*}
where $ \widetilde{\Xi} = C \, \Xi = m \, (\Xi^*)^{-1} \, G^{-1} $. The drawback of this approach is that all the eigenvectors of the matrix $ M_\textsub{K} $ need to be computed, which -- for a large number of basis functions -- might be prohibitively time-consuming. We are often only interested in the leading eigenfunctions.

\subsection{EDMD for the Perron--Frobenius operator}

EDMD as presented in Section~\ref{sec:Numerical approximation} can also directly be used to compute an approximation of the eigenfunctions of the Perron--Frobenius operator. Since
\begin{equation*}
    \widetilde{a}_{ij} = \innerprod{\K \psi_i}{\psi_j}_\mu = \innerprod{\psi_i}{\mathcal{P} \psi_j}_\mu,
\end{equation*}
the entries of the matrix $ \widetilde{A}^T $ are given by $ \innerprod{\mathcal{P} \psi_i}{\psi_j}_\mu $. The matrices $ A $ and $ G $ are approximations of $ \widetilde{A} $ and $ \widetilde{G} $, respectively. Thus, the eigenfunctions of the Perron--Frobenius operator can be approximated by computing the eigenvalues and left eigenvectors of
\begin{equation}
    M_\textsub{P} = A^T \, G^+.
    \label{eq:M_P}
\end{equation}
Analogously, the generalized eigenvalue problem
\begin{equation*}
    \widetilde{\xi} \, A^T = \lambda \, \widetilde{\xi} \, G
\end{equation*}
can be solved. We discuss an even more general way of approximating the adjoint operator in Appendix~\ref{app:ad_EDMD}.

\begin{example}
Let us compute the dominating eigenfunction of the Perron--Frobe\-nius operator for the linear system introduced in Example~\ref{ex:Linear example}. Note that the origin is a fixed point and we would expect the invariant density to be the Dirac distribution $ \delta $ with center $ (0, 0) $. Using monomials of order up to 10 and thin plate splines of the form $ \psi(r) = r^2 \ln r $, where $ r $ is the distance between the point $ (x, y) $ and the center, respectively, we obtain the approximations shown in Figure~\ref{fig:Linear example P1}. This illustrates that the results strongly depend on the basis functions chosen. EDMD will return only meaningful results if the eigenfunctions can be represented by the selected basis.

\begin{figure}[htb]
    \newcommand*{\factor}{0.33}
    \centering
    \begin{minipage}[t]{\factor\textwidth}
        \centering
        \subfiguretitle{a)}
        \includegraphics[width=\textwidth]{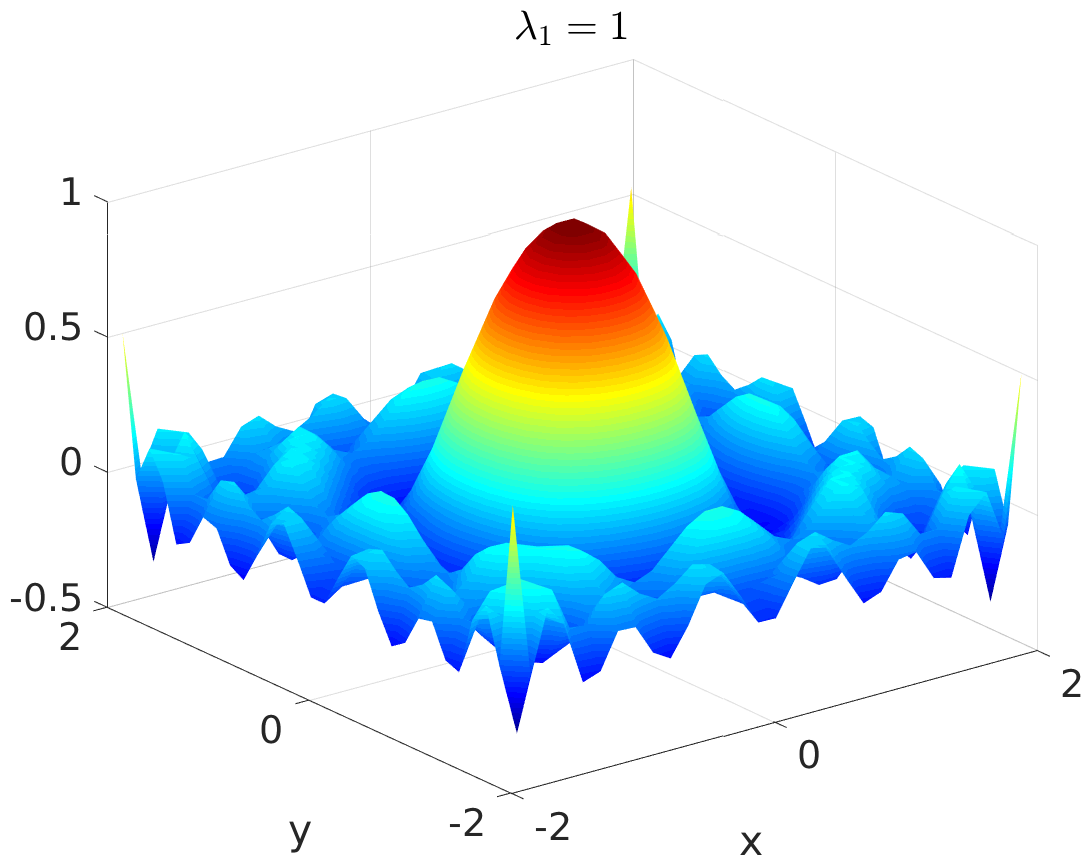}
    \end{minipage}
    \hspace*{1cm}
    \begin{minipage}[t]{\factor\textwidth}
        \centering
        \subfiguretitle{b)}
        \includegraphics[width=\textwidth]{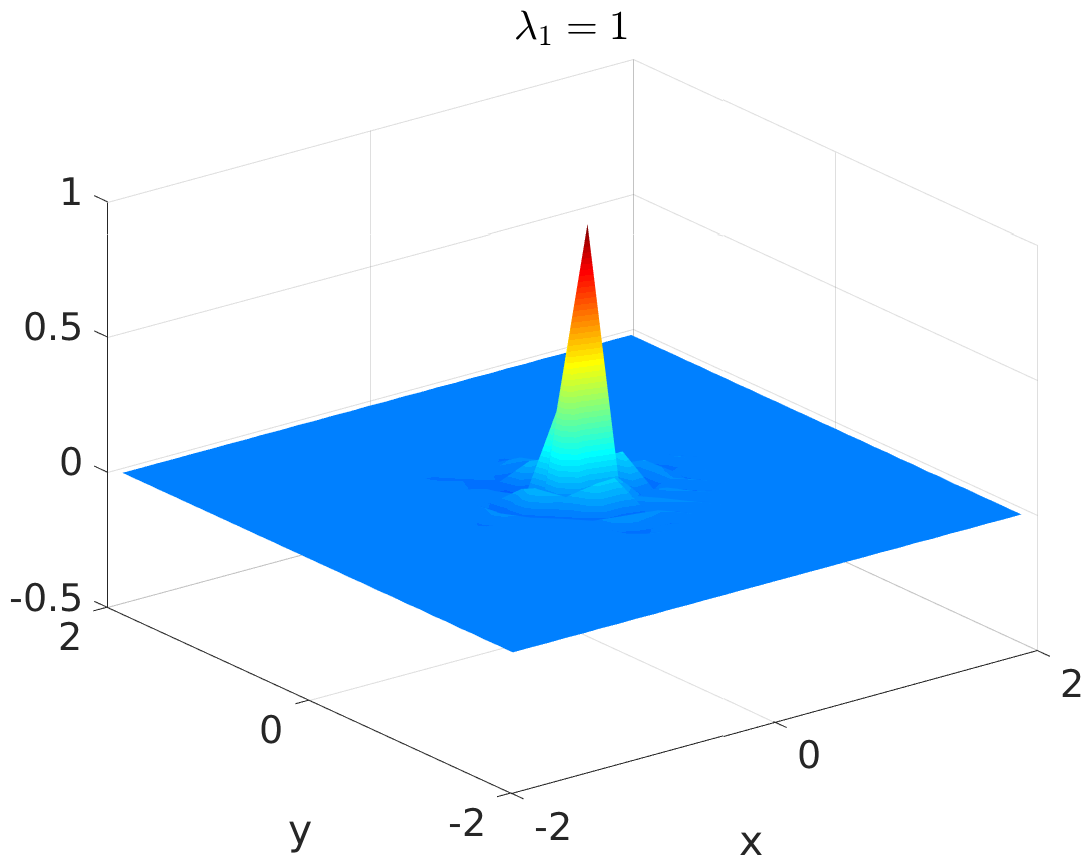}
    \end{minipage}
    \caption{Approximation of the invariant density of the linear system from Example~\ref{ex:Linear example} using a) monomials of order up to 10 and b) thin plate splines. This example shows that EDMD captures the eigenfunctions only if they can be represented by the basis chosen.}
    \label{fig:Linear example P1}
\end{figure}

One possibility to detect whether the chosen basis is sufficient to approximate the dominant eigenfunctions accurately is to add additional basis functions and to check whether the results remain essentially unchanged. Here, one should take into account that the condition number of the problem might deteriorate if a large number of basis functions is used. Another possibility is to compute the residual~$\norm{\Psi_Y - K^T \Psi_X}_F$. A large error indicates that the set of basis functions cannot represent the eigenfunctions accurately. \exampleSymbol

\end{example}

\begin{remark}
The eigenfunctions computed in the previous subsection are identical to the ones computed here. The matrix $ \Xi $ can be computed as the solution of the following generalized eigenvalue problem $ \Xi \, A = \Lambda \, \Xi \, G $. Hence, we get $ A^T = G^T \, \Xi^* \, \Lambda^* \, (\Xi^*)^{-1} $. Then $ \widetilde{\Xi} = (\Xi^*)^{-1} \, G^{-1} $, neglecting the factor $ m $, solves the generalized eigenvalue problem
\begin{equation*}
    \widetilde{\Xi} \, A^T
        = (\Xi^*)^{-1} \, G^{-1} \, G^T \, \Xi^* \, \Lambda^* \, (\Xi^*)^{-1}
        = \Lambda^* \, (\Xi^*)^{-1}
        = \Lambda^* \, \widetilde{\Xi} \, G,
\end{equation*}
using the fact that $ G $ is symmetric.
\end{remark}

This shows that the eigenfunctions of the Koopman operator are approximated by the left eigenvectors and the eigenfunctions of the Perron--Frobenius operator by the right eigenvectors of the generalized eigenvalue problem with the matrix pencil given by $ (A, G) $. The advantage of this approach is that arbitrary basis functions can be chosen to compute eigenfunctions of the Perron--Frobenius operator. This might be beneficial if the eigenfunctions can be approximated by a small number of smooth functions -- for instance monomials, Hermite polynomials, or radial basis functions -- whereas using Ulam's method a large number of indicator functions would be required.

\section{Numerical examples}
\label{sec:Examples}

In this section, we will illustrate the different methods described in the paper using simple stochastic differential equations and molecular dynamics examples.

\subsection{Double-well problem}

Consider the following stochastic differential equation
\begin{equation} \label{eq:Double well}
    \begin{split}
        d\bm x_t &= -\grad_x V(\bm x_t, \bm y_t) \, dt + \sigma \, d\bm w_{t,1}, \\
        d\bm y_t &= -\grad_y V(\bm x_t, \bm y_t) \, dt + \sigma \, d\bm w_{t,2},
    \end{split}
\end{equation}
where $\bm w_{t,1}$ and $ \bm w_{t,2} $ are two independent standard Wiener processes. In this example, the potential, shown in Figure~\ref{fig:Double well}a, is given by $ V(x, y) = (x^2 - 1)^2 + y^2 $ and $ \sigma = 0.7 $.

\begin{figure}[htb]
    \centering
    \begin{minipage}{0.45\textwidth}
        \centering
        \subfiguretitle{a)}
        \includegraphics[width=0.95\textwidth]{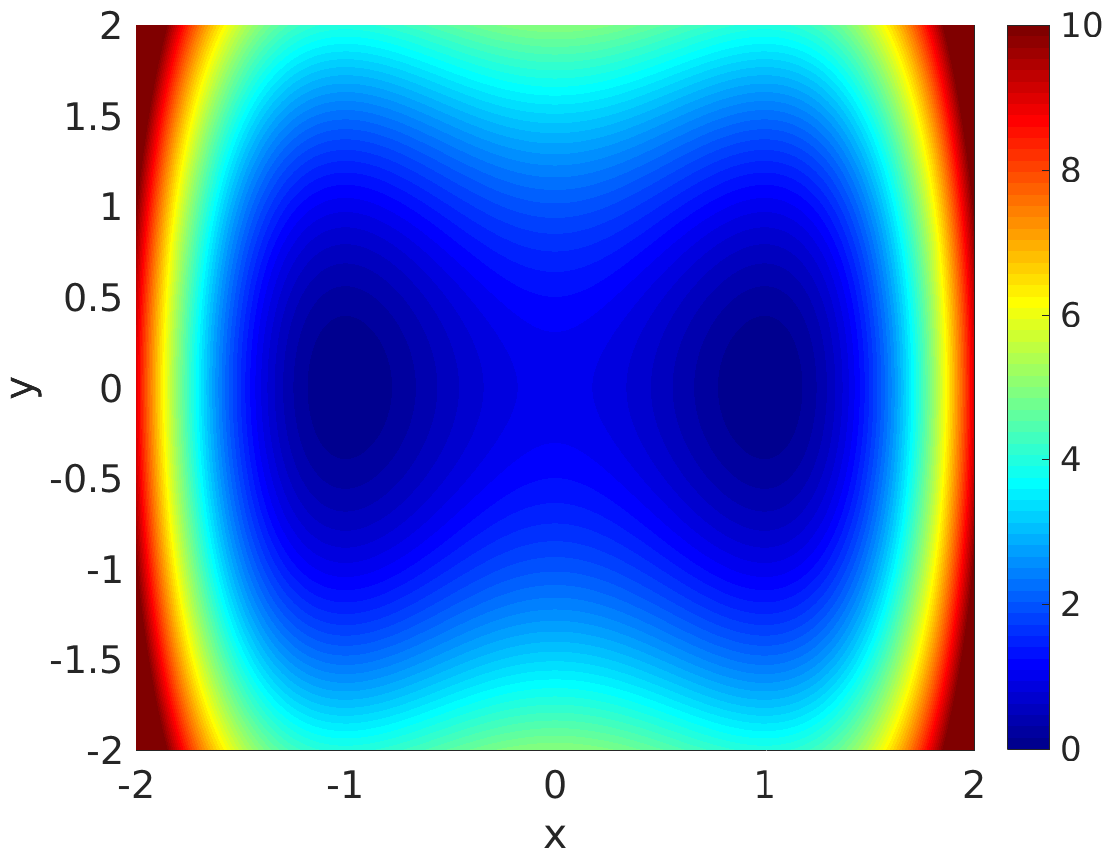}
    \end{minipage}
    \begin{minipage}{0.45\textwidth}
        \centering
        \subfiguretitle{b)}
        \includegraphics[width=0.9\textwidth]{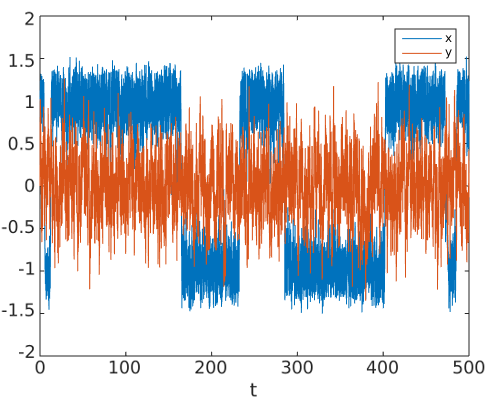}
    \end{minipage}
    \caption{a) Double-well potential $ V(x, y) = (x^2 - 1)^2 + y^2 $. The $ x $ variable typically stays for a long time close to $ x = -1 $ or $ x = 1 $ and rarely switches from one state to the other. The $ y $ variable oscillates around the equilibrium $ y = 0 $. b) Numerical solution of the double-well SDE \eqref{eq:Double well}.}
    \label{fig:Double well}
\end{figure}

Numerically, this system can be solved using the Euler--Maruyama method, which, for an SDE of the form
\begin{equation*}
    d\bm x_t = \mu(t, \bm x_t) \, dt + \sigma(t, \bm x_t) \, d\bm w_t,
\end{equation*}
can be written as
\begin{equation*}
    \bm x_{k+1} = \bm x_k + h \, \mu(t_k, \bm x_k) + \sigma(t_k, \bm x_k) \, \Delta\bm w_k,
\end{equation*}
where $ h $ is the step size and $ \Delta\bm w_k = \bm w_{k+1} - \bm w_k \sim \mathcal{N}(0, h) $. Here, $ \mathcal{N}(0, h) $ denotes a normal distribution with mean~$ 0 $ and variance~$ h $. 
A typical trajectory of system~\eqref{eq:Double well} is shown in Figure~\ref{fig:Double well}b.

\begin{figure}[p]
    \newcommand*{\factor}{0.33}
    \centering
    \subfiguretitle{a)}
    \includegraphics[width=\factor\textwidth]{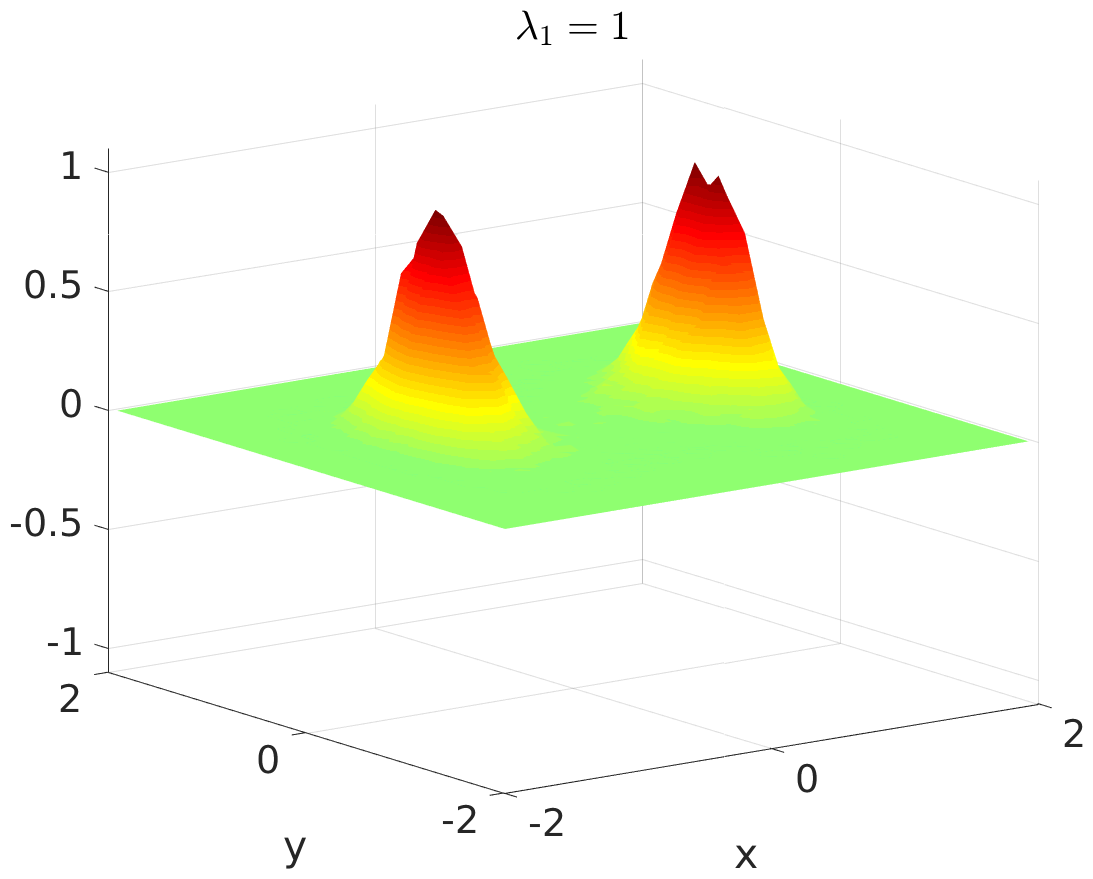} \hspace*{1cm}
    \includegraphics[width=\factor\textwidth]{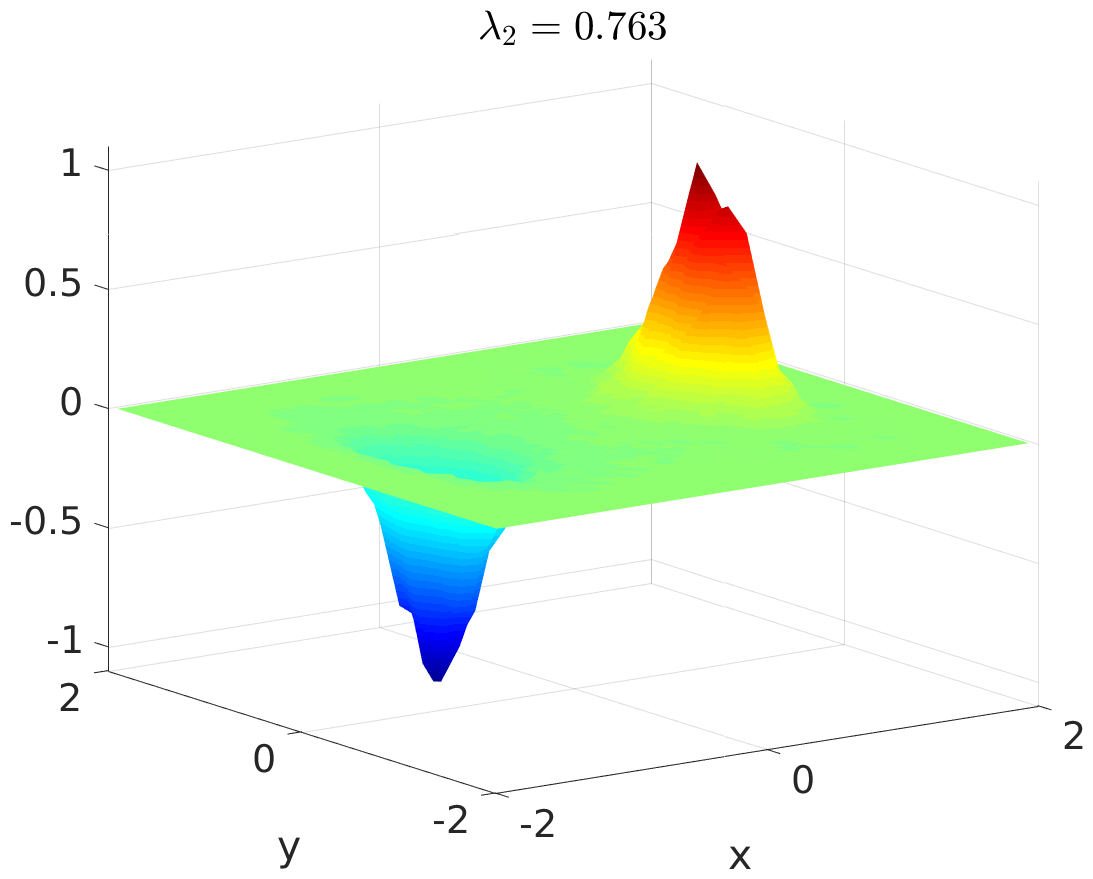} \\
    \subfiguretitle{b)}
    \includegraphics[width=\factor\textwidth]{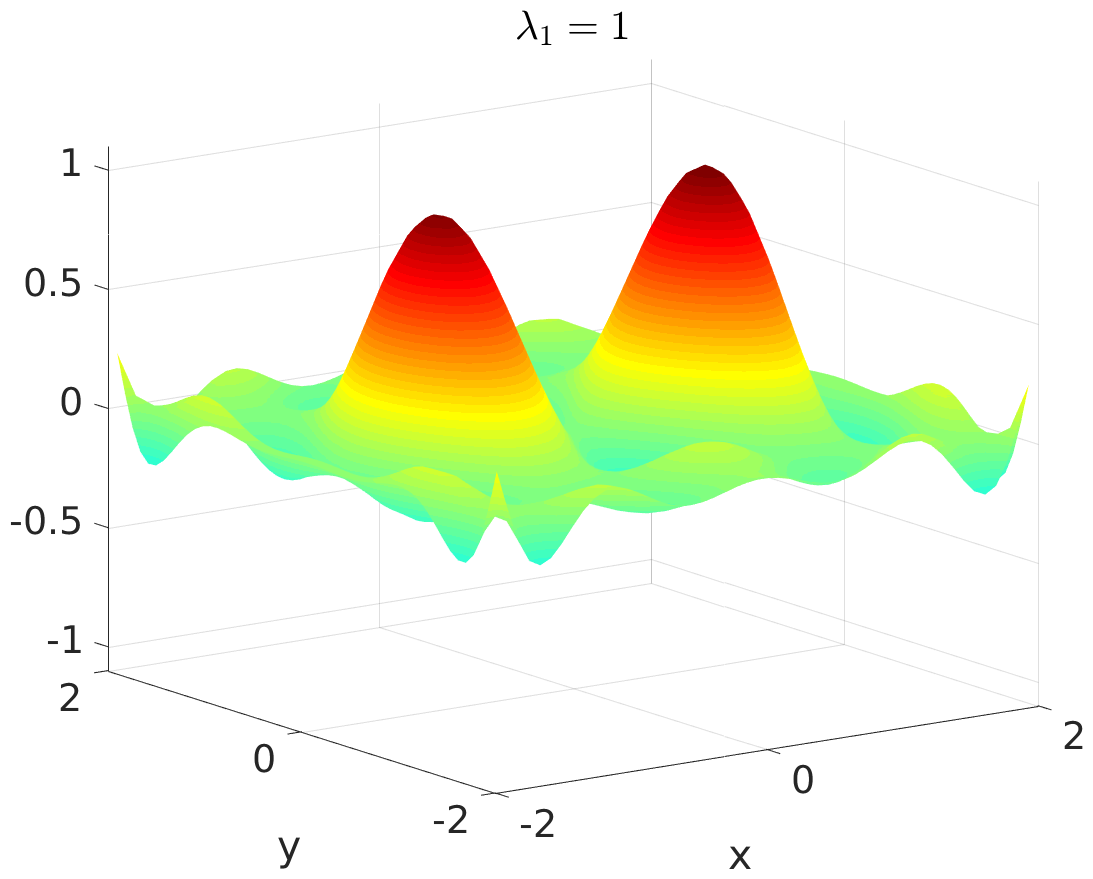} \hspace*{1cm}
    \includegraphics[width=\factor\textwidth]{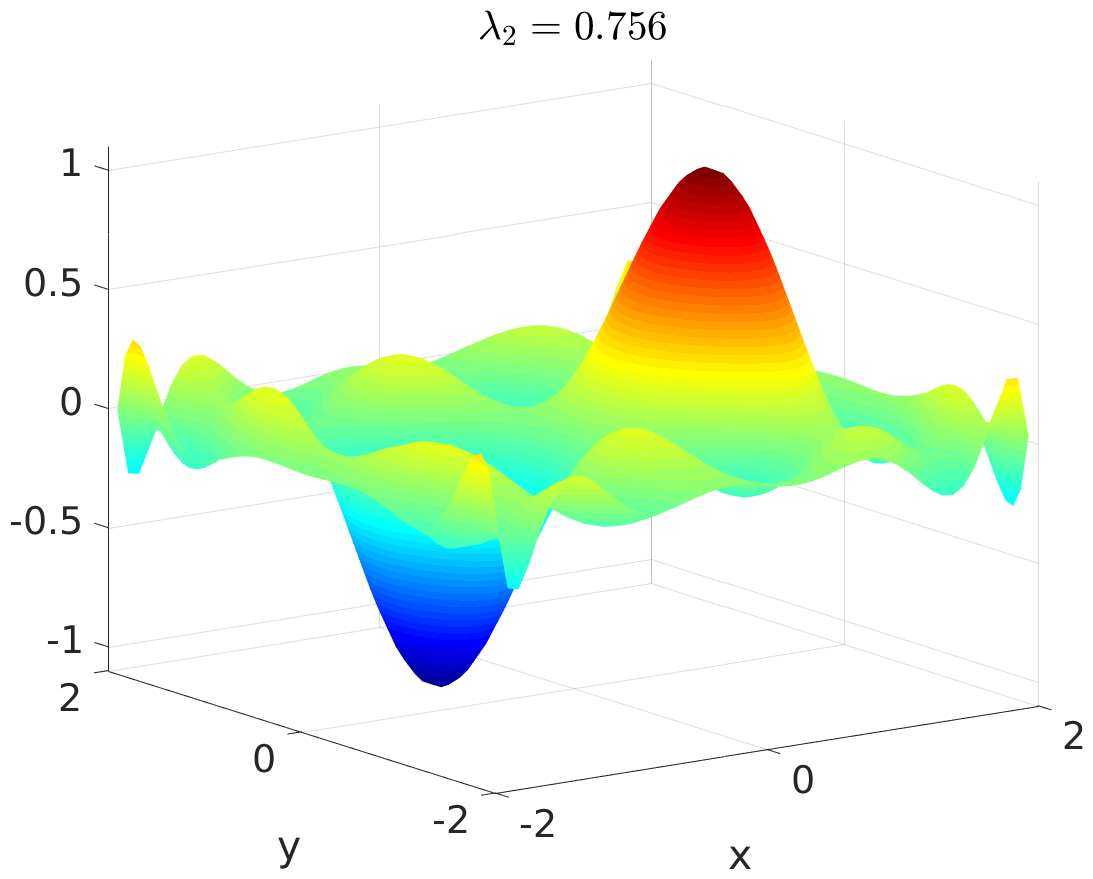} \\
    \subfiguretitle{c)}
    \includegraphics[width=\factor\textwidth]{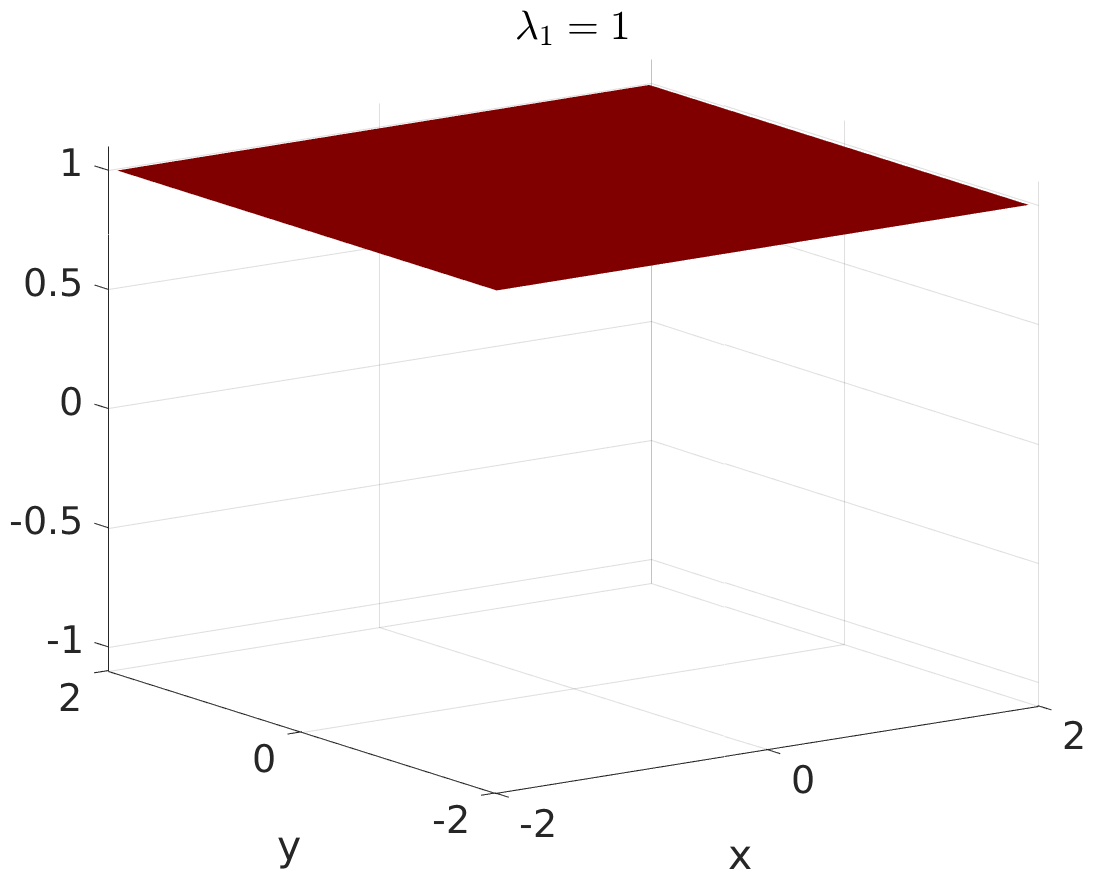} \hspace*{1cm}
    \includegraphics[width=\factor\textwidth]{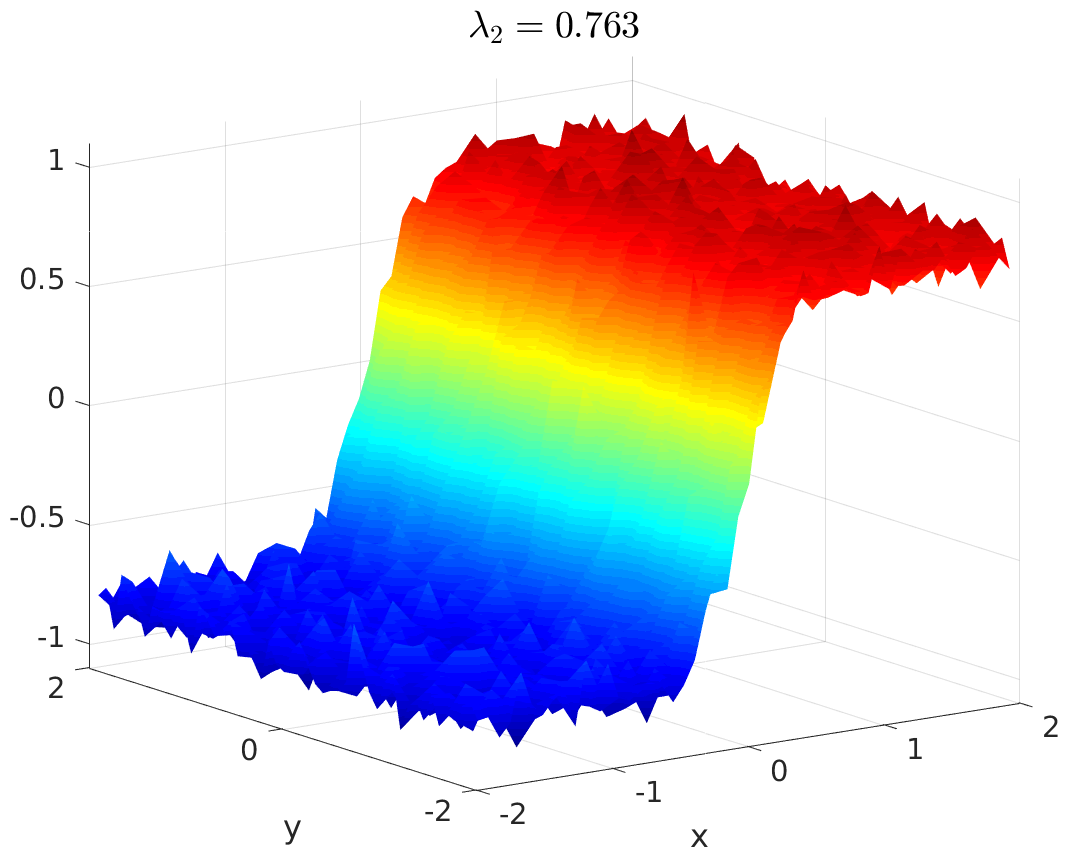} \\
    \subfiguretitle{d)}
    \includegraphics[width=\factor\textwidth]{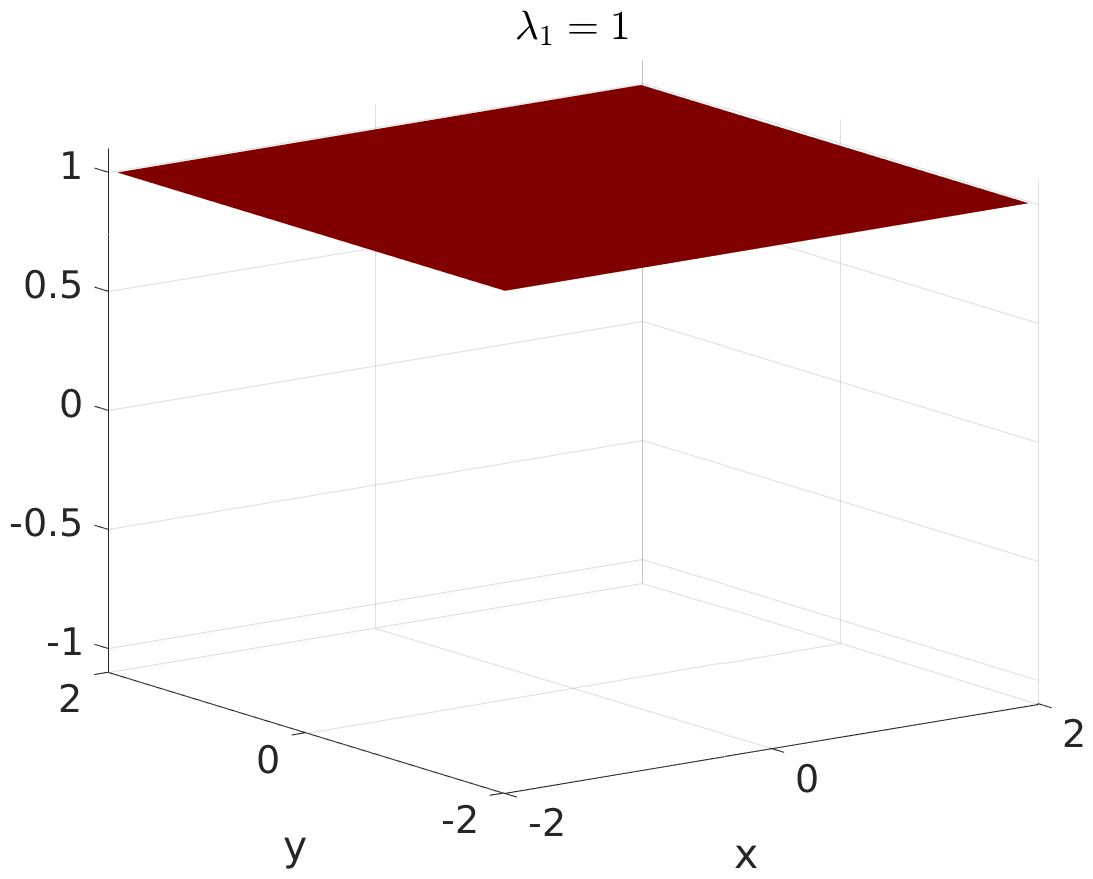} \hspace*{1cm}
    \includegraphics[width=\factor\textwidth]{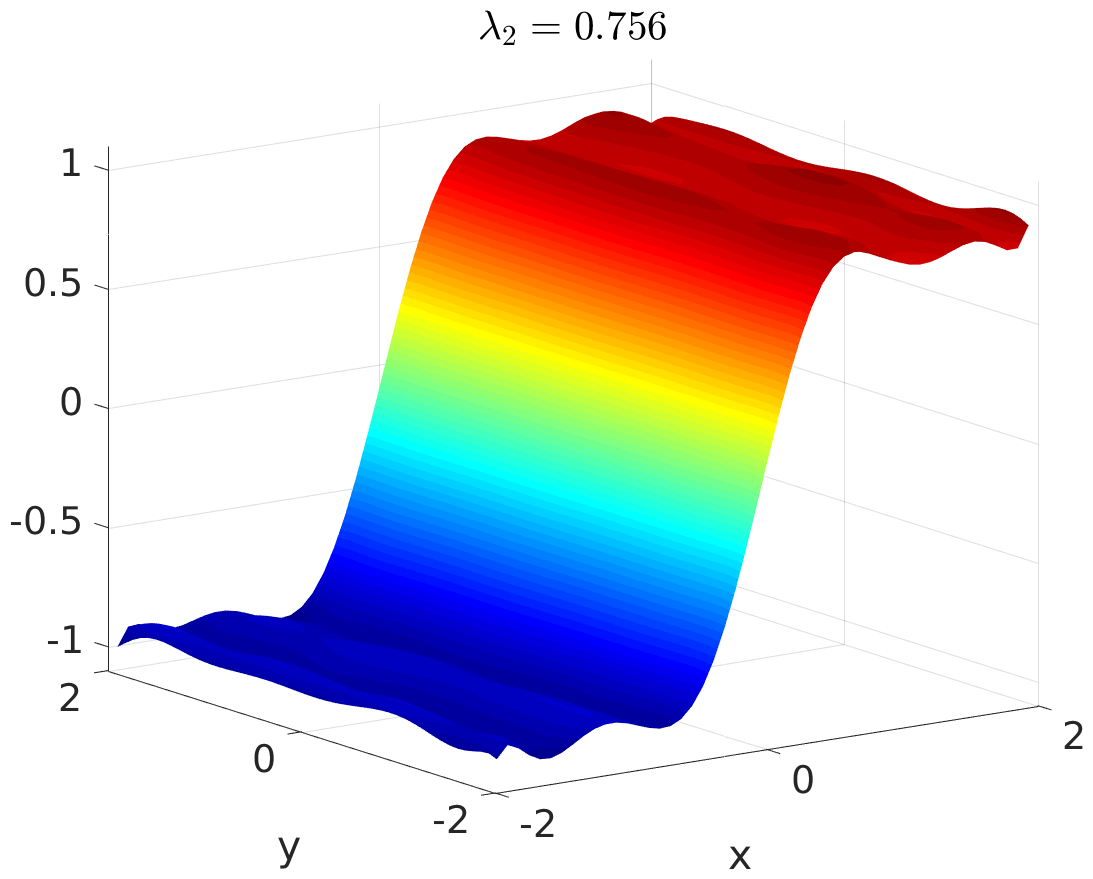}
    \caption{The first two eigenfunctions of the Perron--Frobenius operator $ \mathcal{P} $ for the double-well problem computed using a) Ulam's method and b) EDMD and the first two eigenfunctions of the Koopman operator $ \K $ computed using c) Ulam's method and d) EDMD.}
    \label{fig:Double well PK12}
\end{figure}

In order to compare the different methods described in the preceding sections, we conputed the leading eigenfunctions with Ulam's method and EDMD. For Ulam's method, we partitioned the domain $ \Omega = [-2, 2]^2 $ into $ 50 \times 50 $ boxes of the same size. For EDMD, we used monomials of order up to and including 10, i.e.
\begin{equation*}
    \mathbb{D} = \{1, \, x, \, y, \, x^2, \, x\,y, \, y^2, \, \dots, \, x^2\,y^8, \, x\,y^9, \, y^{10} \}.
\end{equation*}
That is, Ulam's method requires 2500 parameters to describe the eigenfunctions while EDMD requires only 66. For each box, we generated $ n = 100 $ test points, i.e.~$ 250000 $ test points overall, and used the same test points also for EDMD resulting in $ \Psi_X, \Psi_Y \in \R^{66 \times 250000} $. The system \eqref{eq:Double well} is solved using the Euler--Maruyama method with a step size of $ h = 10^{-3} $. One evaluation of the corresponding dynamical system~$ \bm\Phi $ corresponds to~$ 10000 $ steps. That is, each initial condition is integrated from~$ t_0 = 0 $ to $ t_1 = 10 $. The first two eigenfunctions of the Perron--Frobenius operator and Koopman operator are shown in Figure~\ref{fig:Double well PK12}. Observe that the computed eigenvalues are -- as expected -- almost identical. The second eigenfunction computed with Ulam's method is still very coarse, increasing the number of test points per box would smoothen the approximation. Since for EDMD only smooth basis functions were chosen, the resulting eigenfunction is automatically smoothened.

The system has two metastable states and the second eigenfunction of the Perron--Frobenius operator can be used to detect these metastable states. Also the second eigenfunction of the adjoint Koopman operator contains information about a possible partitioning of the state space, it is almost constant in the $ y $-direction and also almost constant in the $ x $-direction except for an abrupt transition from $ -1 $ to $ 1 $ between the two metastable sets. The other eigenvalues of the system are numerically zero.

\subsection{Triple-well problem}

Consider the slightly more complex triple-well potential
\begin{equation} \label{eq:Triple well potential}
    V(x, y) = 3 \, e^{-x^2-(y - \frac{1}{3})^2} - 3 \, e^{-x^2-(y - \frac{5}{3})^2}
            - 5 \, e^{-(x-1)^2-y^2}             - 5 \, e^{-(x+1)^2-y^2}
            + \tfrac{2}{10} \, x^4              + \tfrac{2}{10} \left(y-\tfrac{1}{3}\right)^4 
\end{equation}
taken from~\cite{SS13}. Here, the variables $ x $ and $ y $ are coupled, i.e.~the potential cannot be written as $ V(x, y) = V_1(x) + V_2(y) $ anymore.
\begin{figure}[htb]
    \centering
    \includegraphics[width=0.4\textwidth]{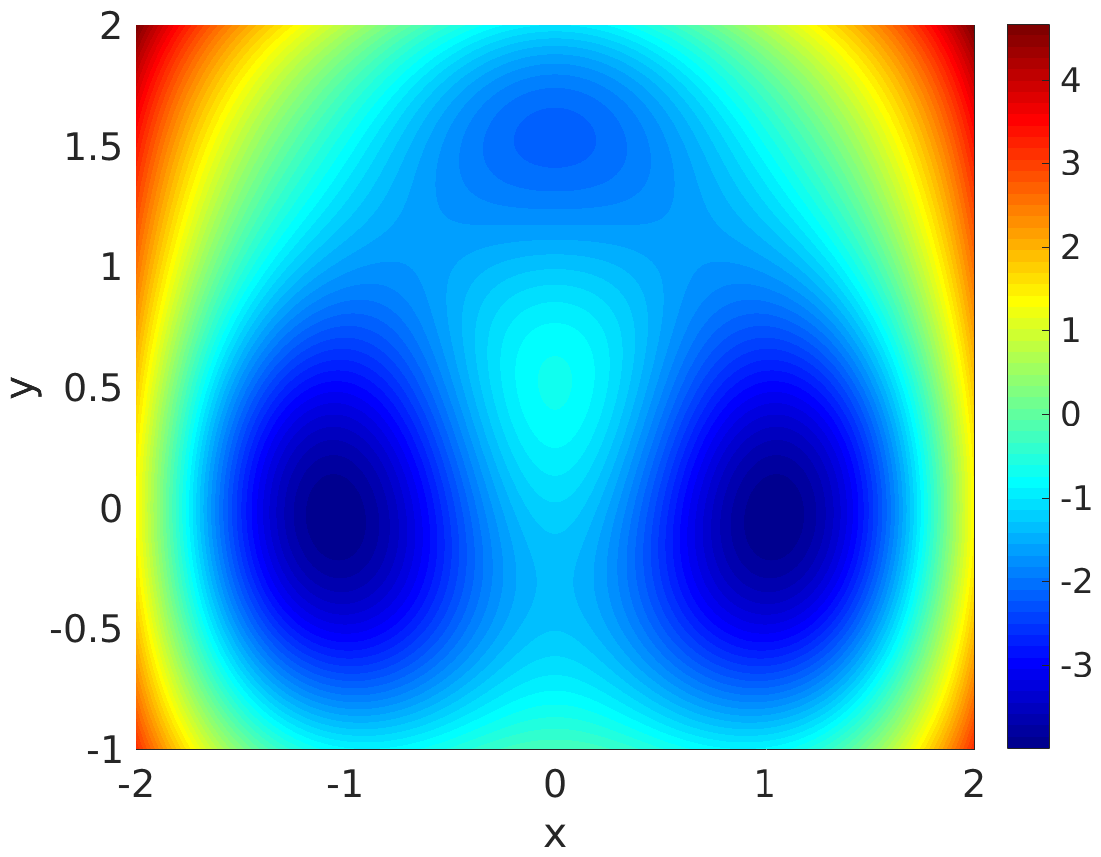}
    \caption{Triple-well potential given by \eqref{eq:Triple well potential}.}
    \label{fig:Triple well potential}
\end{figure}
The potential function is shown in Figure~\ref{fig:Triple well potential} and the first two nontrivial eigenfunctions of the Perron--Frobenius operator and the Koopman operator in Figure~\ref{fig:Triple well PK23}. Note that the eigenfunction $ \varphi_2 $ separates the two deep wells at $ (-1, 0) $ and $ (1, 0) $ and is near zero for the well at $ (0, 1.5) $, the third eigenfunction $ \varphi_3 $ separates the two deep wells from the shallow well. Here, the eigenfunctions of the Perron--Frobenius operator and the eigenfunctions of the Koopman operator essentially encode the same information. As before, we used EDMD with monomials of order up to and including $ 10 $ and $ 250000 $ randomly generated test points within the domain $ \Omega = [-2, 2] \times [-1, 2] $. Each test point was integrated from $ t_0 = 0 $ to $ t_1 = 0.1 $ using a step size of $ h = 10^{-5} $. The parameter $ \sigma $ was set to $ 1.09 $.

\begin{figure}[htb]
    \newcommand*{\factor}{0.33}
    \centering
    \subfiguretitle{a)}
    \includegraphics[width=\factor\textwidth]{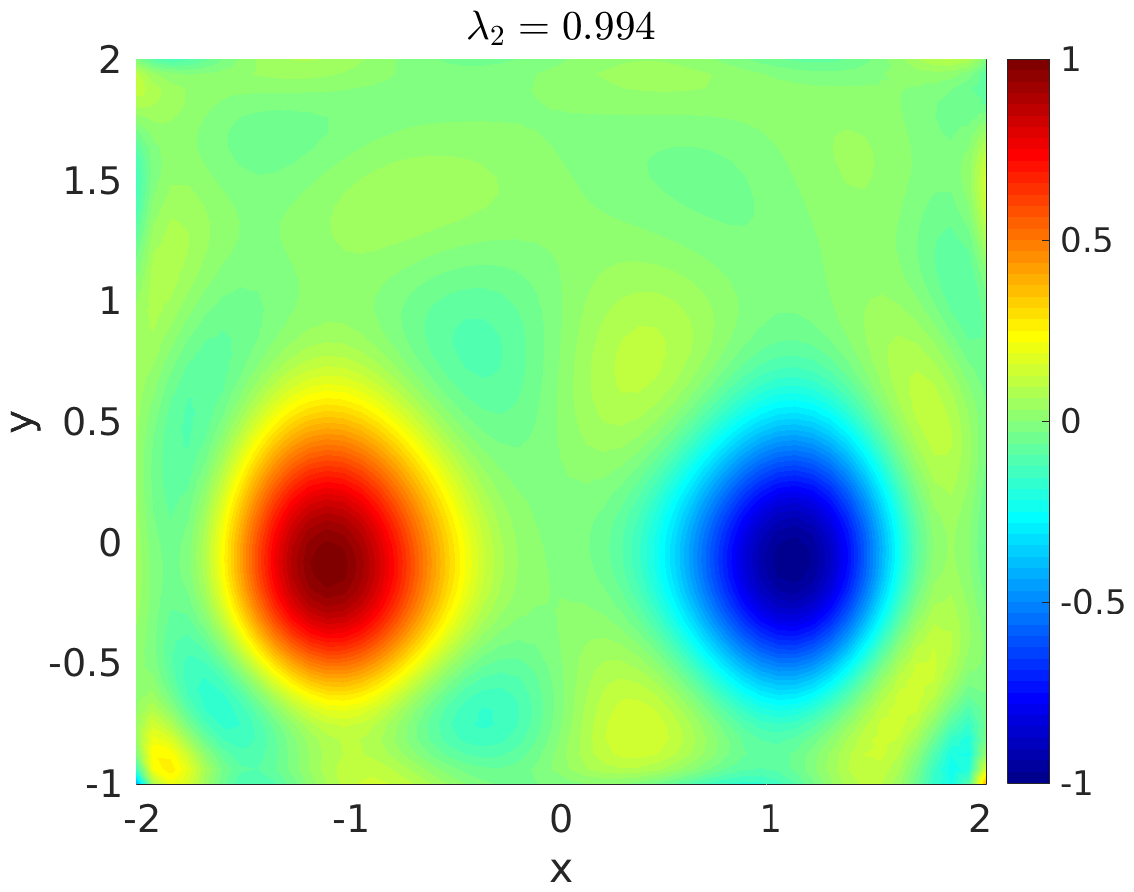} \hspace*{1cm}
    \includegraphics[width=\factor\textwidth]{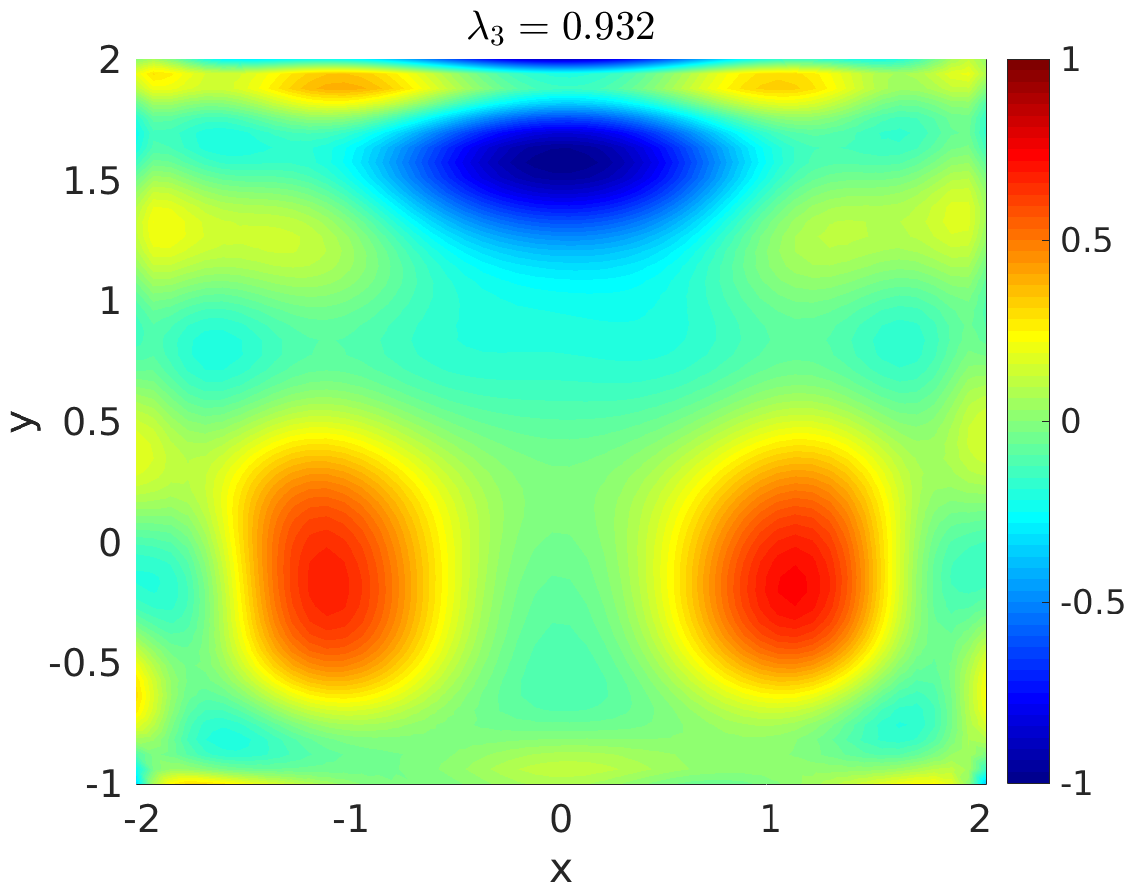} \\
    \subfiguretitle{b)}
    \includegraphics[width=\factor\textwidth]{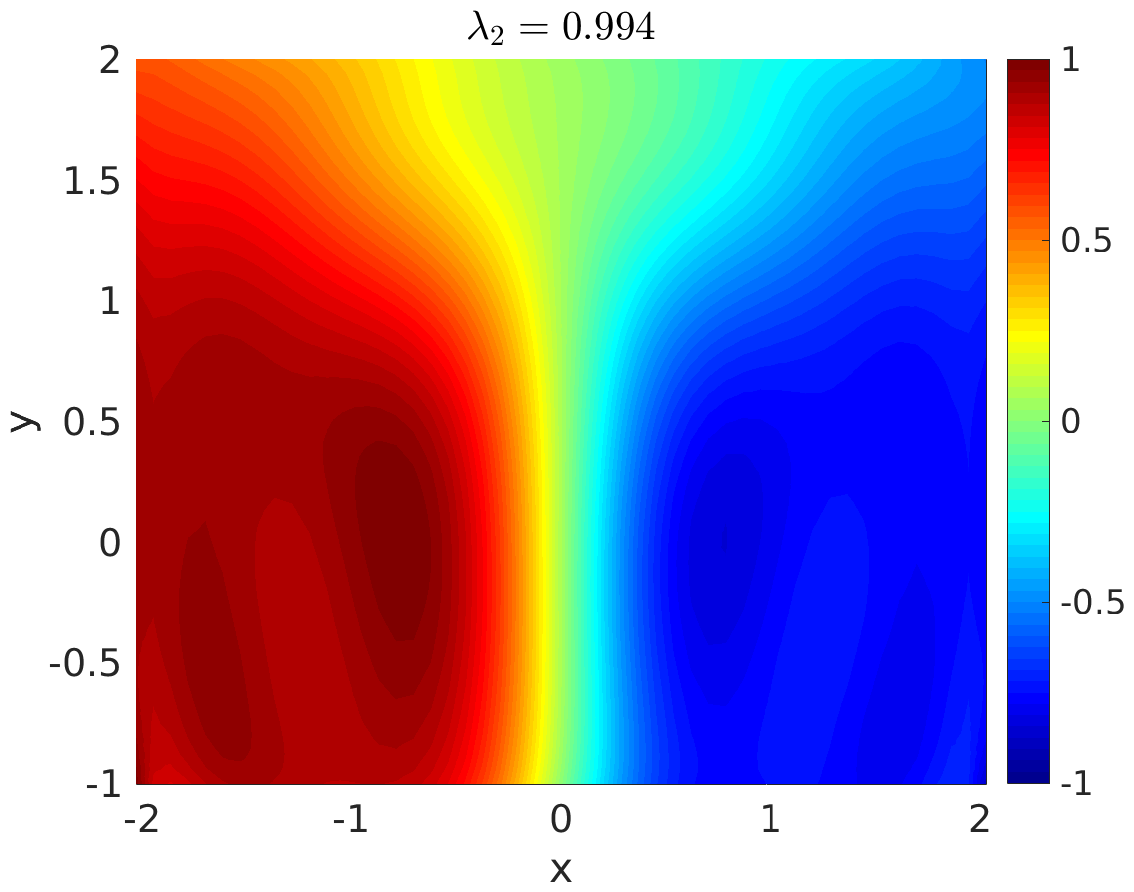} \hspace*{1cm}
    \includegraphics[width=\factor\textwidth]{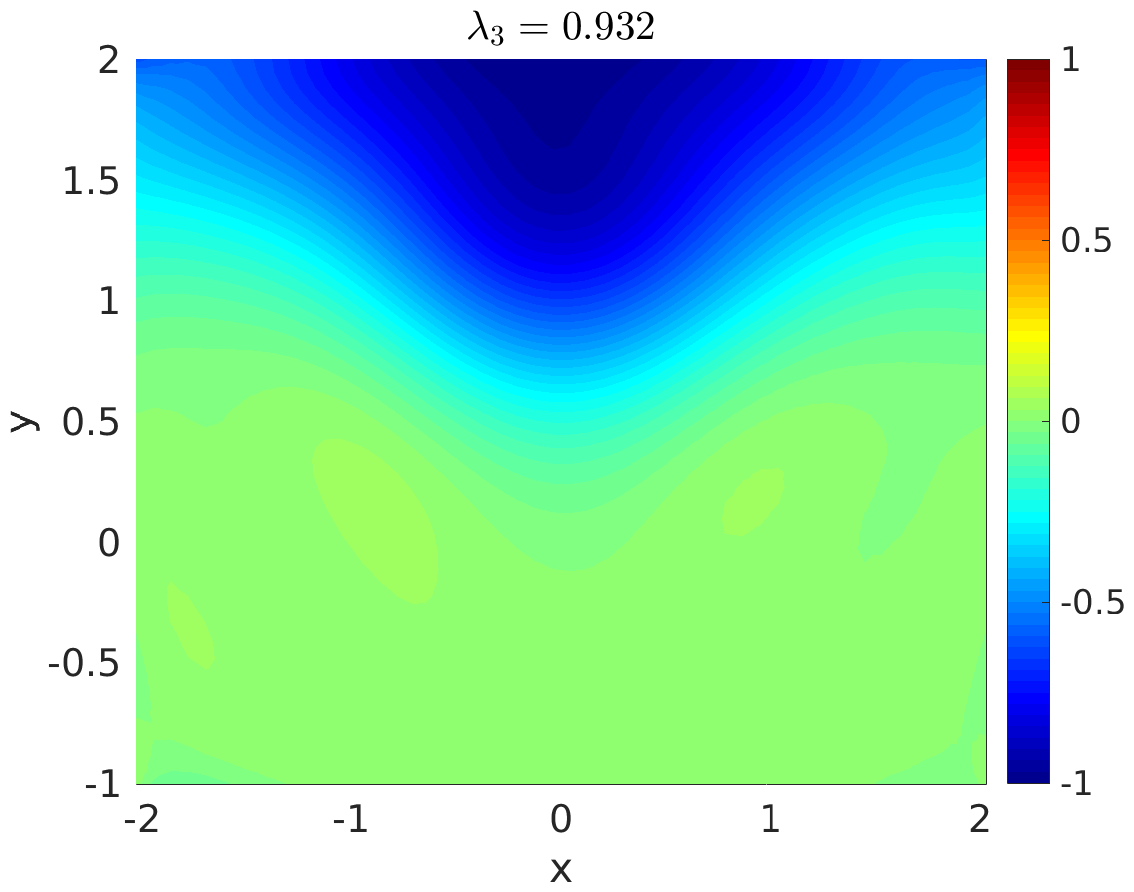}    
    \caption{Second and third eigenfunction of a) the Perron--Frobenius operator $ \mathcal{P} $ and b) the Koopman operator $ \K $ for the triple-well problem.}
    \label{fig:Triple well PK23}
\end{figure}

\subsection{Molecular dynamics and conformation analysis}

\paragraph{Classical molecular dynamics.}

Classical molecular dynamics describes the motion of atoms, or groups of atoms, in terms of Hamiltonian dynamics under the influence of atomic interaction forces resulting from a potential. The \emph{position} or \emph{configuration} space~$\mathbb{Q}\subset\R^d$ describes all possible positions of the atoms, while the momentum space~$\mathbb{P}=\R^d$ contains all momenta. The potential~$V:\mathbb{Q}\to\R$ is assumed to be a sufficiently smooth function. The \emph{phase} space~$\mathbb{X} = \mathbb{Q}\times\mathbb{P}$ of the molecule consists of all possible position-momenta pairs~$x=(q,p)$. The evolution of a molecule in phase space under ideal conditions is described by Hamilton's equations of motion
\begin{equation*} 
    \begin{aligned}
        \dot{q}_t &= M^{-1}p_t, 
    \end{aligned}
\end{equation*}  \begin{equation} \label{eq:Hamilton}
    \begin{aligned}      \dot{p}_t &= -\nabla V(q_t),
    \end{aligned}
\end{equation}
where~$M$ denotes the symmetric positive definite mass matrix. Since molecules do not stand alone, but are rather subject to interaction with their surrounding molecules, different models incorporating these interactions are more commonly used. One way to account for the collisions with the surrounding molecules is to include a damping and a stochastic forcing term in~\eqref{eq:Hamilton} to obtain the \emph{Langevin equation}
\begin{equation} \label{eq:Langevin}
    \begin{aligned}
        d\bm q_t &= M^{-1}\bm p_t dt, \\
        d\bm p_t &= -\nabla V(\bm q_t)dt - \gamma M^{-1}\bm p_t dt+ \sigma d\bm w_t.
    \end{aligned}
\end{equation}
This is an SDE giving rise to a non-deterministic evolution, hence positions and momenta are random variables. Here,~$\bm w_t$ is a standard Wiener process in~$\R^d$. Further,~$\gamma$ and~$\sigma$ satisfy the fluctuation-dissipation relation~$2\gamma = \beta\sigma\sigma^T$, where~$0<\beta$ is called the inverse temperature. This is due to the fact that~$\beta = (k_{\rm B}T)^{-1}$, where~$T$ is the macroscopic temperature of the system, and~$k_{\rm B}$ is the Boltzmann constant. The fluctuation-dissipation relation ensures that the energy of the system is conserved in expectation.

It can also be shown (cf.~\cite{MaSt02,SS13}) that the Langevin process, governed by~\eqref{eq:Langevin}, has a unique invariant density with respect to which it is ergodic. This density is also called the \emph{canonical} density, and has the explicit form~$f_{\rm can}(q,p) = Z^{-1} \exp\left(- \beta (\tfrac12 p^T Mp + V(q))\right)$, where~$Z$ is a normalizing constant.

\paragraph{Spatial transfer operator.}

One of the main features of molecules we are interested in is that it has several important geometric forms, called \emph{conformations}, between which it ``switches''. Hereby it spends ``long'' times (measured on the time scales of its internal motion) in one conformation, and passes quickly to another. Due to this time scale separation the conformations are called \emph{metastable}. The identification of metastable conformations is of major interest, and it is connected to the sub-dominant eigenfunctions of a special transfer operator which is adapted to the problem at hand~\cite{Sch99}: although the more appreciated models describe the dynamics of a molecule in the complete phase space including positions and momenta, metastability is observed (and described) in the positional coordinate only.

This problem-adapted transfer operator is called the \emph{spatial transfer operator} (cf.~\eqref{eq:spatial} below), and describes positional density fluctuations in a molecule which is in thermal equilibrium. More precisely, if an ensemble of molecules with positional coordinates distributed according to the density~$w:\mathbb{Q}\to\R$ with respect to the canonical distribution is given, then its image under the spatial transfer operator with lag time~$t$ describes the density of the positional coordinate of the ensemble after time~$t$, again with respect to the canonical distribution:
\begin{equation} \label{eq:spatial}
    \Sp^t w(q) = \frac{1}{f_{\mathbb{Q}}(q)}\int_{\mathbb{P}} \left(\P^t_{\rm Lan} wf_{\rm can}\right)(q,p)\,dp,
\end{equation}
where~$f_{\mathbb{Q}}$ is the positional marginal of the canonical density, i.e.
\begin{equation*}
    f_{\mathbb{Q}}(q) = \int_{\mathbb{P}}f_{\rm can}(q,p)\,dp,
\end{equation*}
and~$\P^t_{\rm Lan}$ is the transfer operator of the Langevin process governed by~\eqref{eq:Langevin}.

The operator~$\Sp^t:L^2(\mathbb{Q},\mu_{\mathbb{Q}})\to L^2(\mathbb{Q},\mu_{\mathbb{Q}})$, where~$d\mu_{\mathbb{Q}}(q) = f_{\mathbb{Q}}(q)dq$, is self-adjoint (i.e.~has pure real point spectrum), and due to the ergodicity of the Langevin process it possesses the isolated and simple eigenvalue~1 with corresponding eigenfunction~$\mathds{1}_{\mathbb{Q}}$~\cite{BiKoJu15}.

With the right chemical intuition at hand the range of positional coordinates possibly interesting for conformation analysis can be drastically reduced to just a handful of \emph{essential coordinates}; as it is shown in Section~\ref{ssec:butane}. The spatial transfer operator can be adapted to this situation, as we describe in Appendix~\ref{app:EDMD spatial essential}. There we also show that if we carry out the EDMD procedure in the space of these reduced observables, we actually approximate a Galerkin projection of the corresponding \emph{reduced} spatial transfer operator. A similar technique has been developed in~\cite{NoNu13,NKPMN14}. Chekroun et al~\cite{Chekroun2014} also approximate a reduced transfer operator from observable time series from climate models, but only for the case where the basis functions are characteristic functions, as in Ulam's method.

\subsection{\texorpdfstring{$n$}{n}-butane} \label{ssec:butane}

Let us now consider the $ n $-butane molecule H$_3$C$-$CH$_2-$CH$_2-$CH$_3$ shown in Figure~\ref{fig:Butane} (drawn with PyMOL~\cite{PyMOL10}). We want to analyze this molecule since the energy landscape and conformations are well-known. 
The four configurations illustrated in Figure~\ref{fig:Butane} can be obtained by rotating around the bond between the second and third carbon atom. The potential energy of a molecule depends on the structure. The higher the potential energy of a conformation, the lower the probability the system will remain in that state. Thus, we would expect a high probability for the anti configuration, a slightly lower probability for the gauche configuration, and low probabilities for the other configurations. Indeed, the anti and gauche configurations are metastable conformations.

\begin{figure}[htb]
    \centering
    \begin{minipage}{0.24\textwidth}
        \centering
        \subfiguretitle{a) $ \varphi = 0\degree $}
        \includegraphics[width=\textwidth]{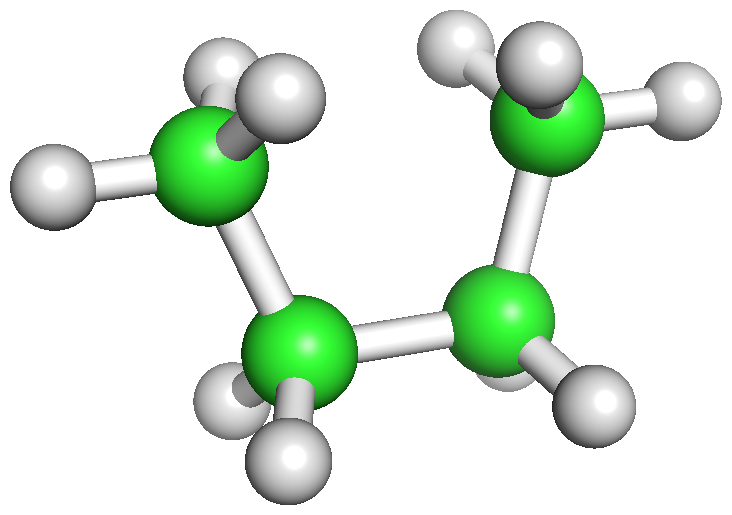} \\
        \subfiguretitle{b) $ \varphi = 60\degree $}
        \includegraphics[width=\textwidth]{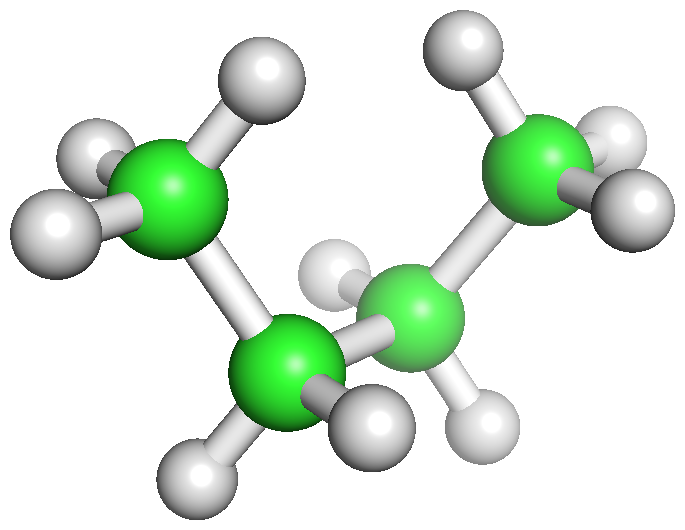}
    \end{minipage}
    \begin{minipage}{0.5\textwidth}
        \centering
        \includegraphics[width=\textwidth]{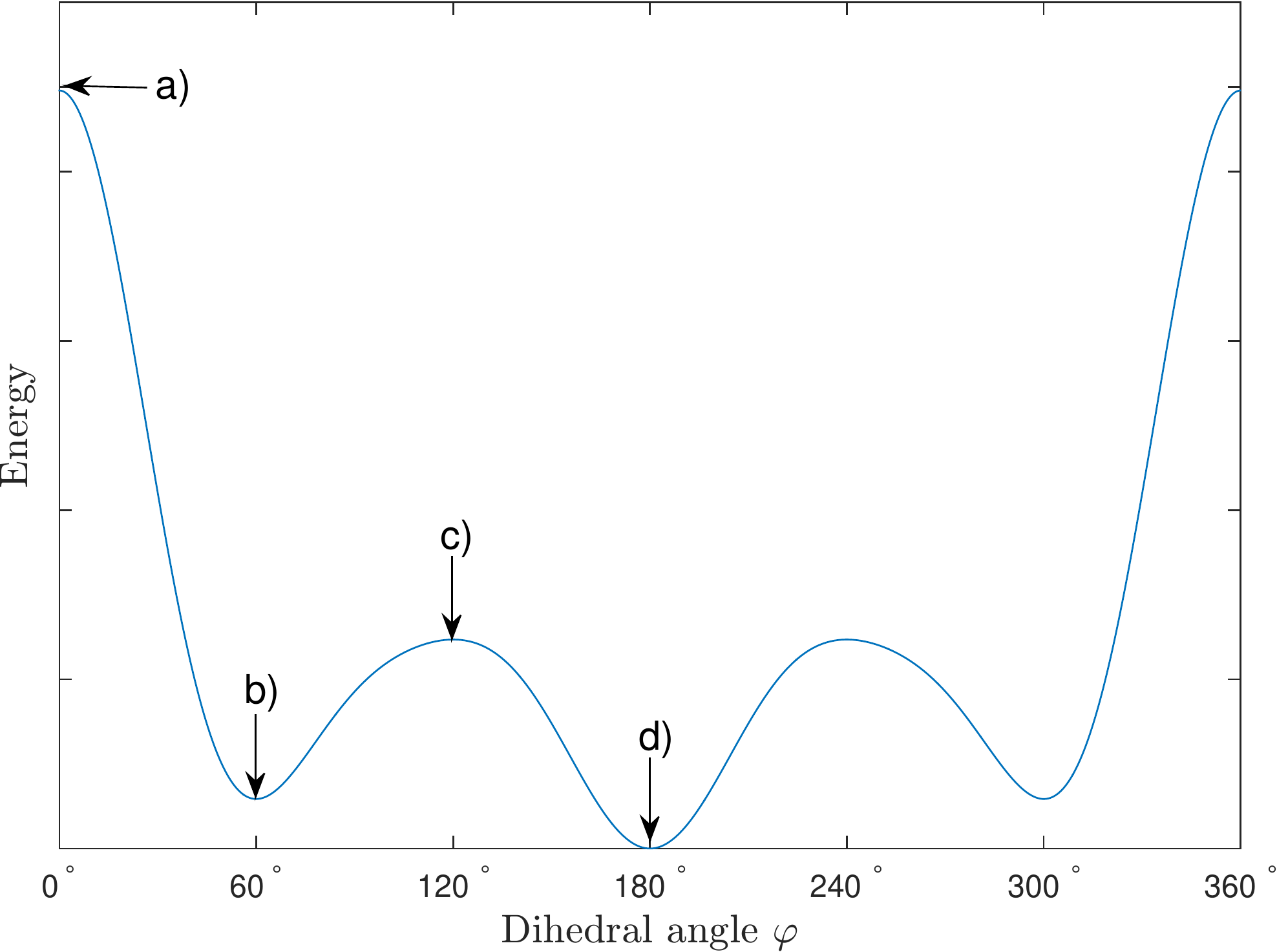}
    \end{minipage}
    \begin{minipage}{0.24\textwidth}
        \centering
        \subfiguretitle{c) $ \varphi = 120\degree $}
        \includegraphics[width=\textwidth]{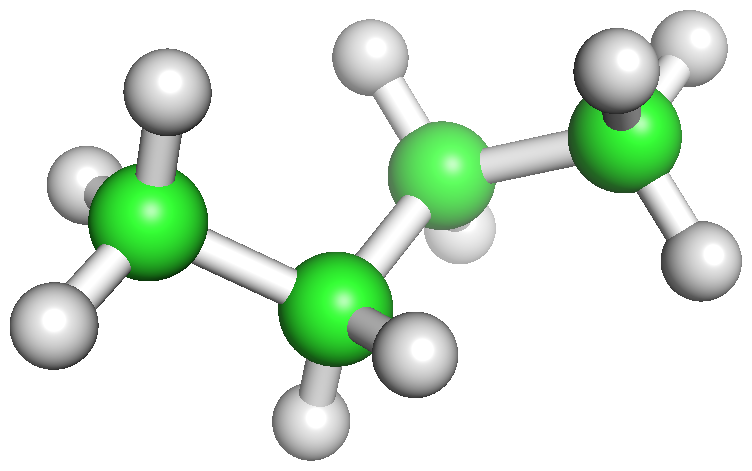} \\
        \subfiguretitle{d) $ \varphi = 180\degree $}
        \includegraphics[width=\textwidth]{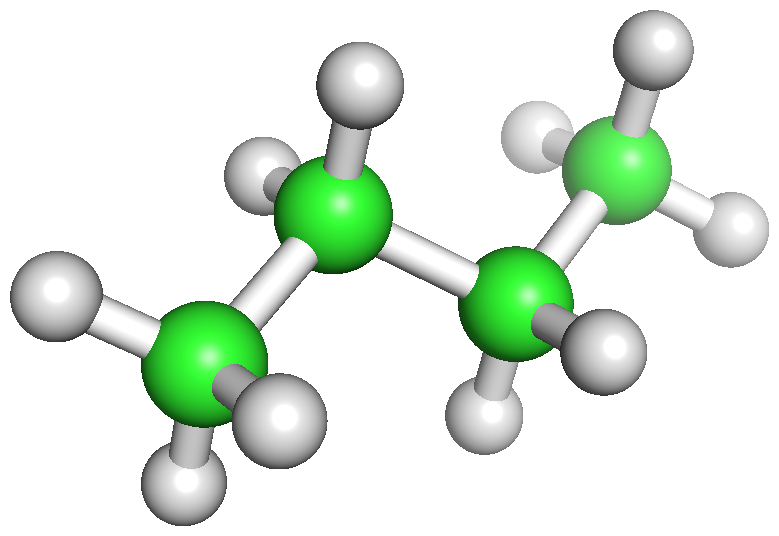}
    \end{minipage}
    \caption{Potential of the $ n $-butane molecule as a function of the dihedral angle  $ \varphi $ and different conformations. a)~Fully eclipsed. b)~Gauche. c)~Partly eclipsed. d)~Anti.}
    \label{fig:Butane}
\end{figure}

\begin{figure}[htb]
    \centering
    \includegraphics[width=0.3\textwidth]{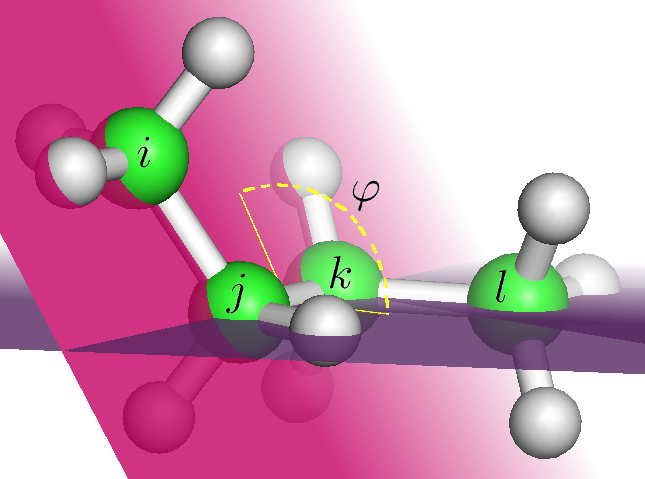}
    \caption{Definition of the dihedral angle $ \varphi $ for the butane molecule.}
    \label{fig:Dihedral angle}
\end{figure}

Molecular dynamics simulators are standard tools to analyze the conformations and conformational dynamics of biological molecules such as proteins and the extraction of this essential information from molecular dynamics simulations is still an active field of research~\cite{NSVN15}. We simulated the $ n $-butane molecule for an interval of $ 10 \unit{ns} $ with a step size of $ 2 \unit{fs} $ using AmberTools15~\cite{Amber15} and, downsampling by a factor of 100, created one trajectory containing 50,000 data points. From this $ 42 $-dimensional trajectory -- 3 coordinates for each of the 14 atoms --, we extracted the dihedral angle $ \varphi $ shown in Figure~\ref{fig:Dihedral angle} as
\begin{equation} \label{eq:dihedral}
    \cos \varphi = \frac{n_1 \cdot n_2}{\norm{n_1} \, \norm{n_2}},
\end{equation}
where $ n_1 = v_{ij} \times v_{jk} $ and $ n_2 = v_{lk} \times v_{jk} $ are the vectors perpendicular to the planes spanned by the carbon atoms $ i, j, k $ and $ j, k, l $, respectively, and $ v_{ij} $ is the bond between $ i $ and $ j $.

In order to compute the dominant eigenfunctions of the spatial transfer operator for this one essential coordinate, we used $ 41 $ basis functions  $ \{ 1, \, \cos(i \, x), \, \sin(i \, x) \} $, $ i = 1, \dots, 20 $, for the interval $ [ 0, 2 \, \pi] $. The resulting leading eigenfunctions are shown in Figure~\ref{fig:ButaneP123}. As expected, the first eigenfunction predicts high probabilites for the gauche and anti configurations and low probabilites for the other configurations. The (sign) structure of the second and third eigenfunctions contain information about the metastable sets.

\begin{figure}[htb]
    \centering
    \includegraphics[width=\textwidth]{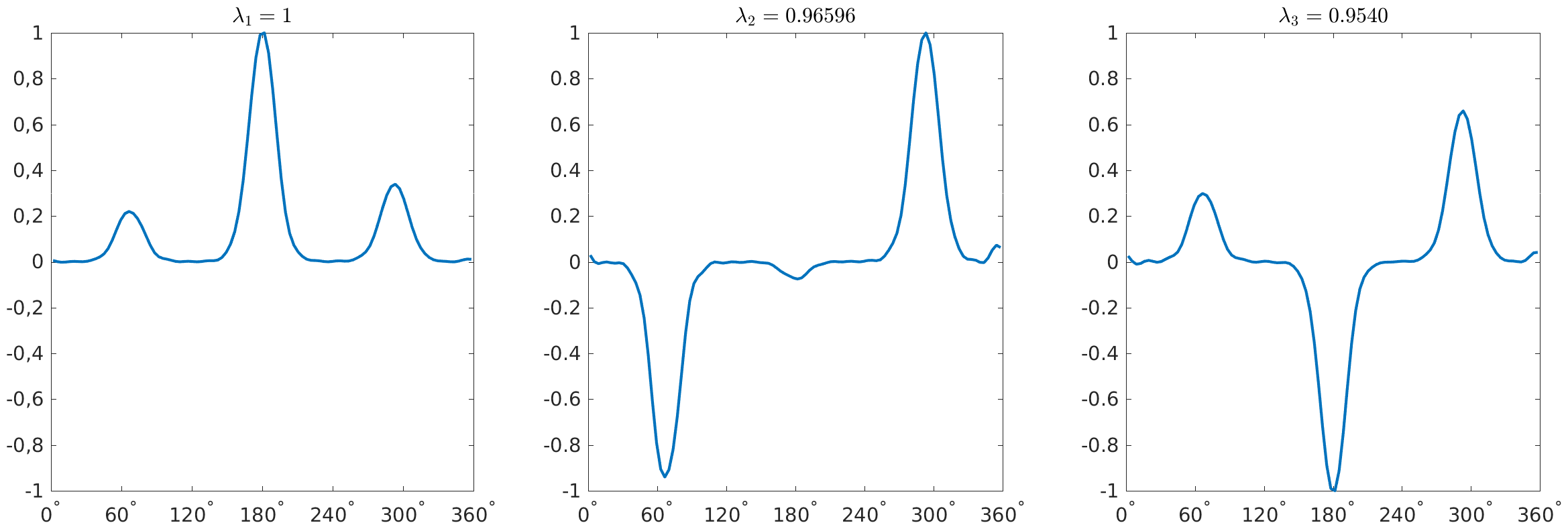}
    \caption{First three eigenfunctions of the Perron--Frobenius operator obtained from a simulation of the $ n $-butane molecule computed with EDMD.}
    \label{fig:ButaneP123}
\end{figure}

\section{Conclusion}
\label{sec:Conclusion}

The global behavior of dynamical systems can be analyzed using operator-based approaches. We reviewed and described different, projection-based numerical methods such as Ulam's method and EDMD to compute finite-dimensional approximations of the Perron--Frobenius operator and the Koopman operator. Furthermore, we highlighted the similarities and differences between these methods and showed that methods developed for the approximation of the Koopman operator can be used for the Perron--Frobenius operator, and vice versa. We demonstrated the performance of different methods with the aid of several examples. If the eigenfunctions of the Perron--Frobenius operator or Koopman operator are smooth, EDMD enables an accurate approximation with a small number of basis functions. Thus, this approach is well suited also for higher-dimensional problems.

The next step could be to investigate the possibility of extending the methods reviewed within this paper using tensors as described in~\cite{NSVN15} for reversible processes. Currently, not all numerical methods required for generalizing these methods to tensor-based methods are available. Nevertheless, developing tensor-based algorithms for these eigenvalue problems might enable the analysis of high-dimensional systems.

\section*{Acknowledgements}

This research has been partially funded by Deutsche Forschungsgemeinschaft (DFG) through grant CRC 1114 and by the Einstein Center for Mathematics Berlin (ECMath), project grant CH7.

\bibliographystyle{abbrv}
\bibliography{PFK}

\begin{thebibliography}{10}

\bibitem{BaRo95}
J.~R. Baxter and J.~S. Rosenthal.
\newblock Rates of convergence for everywhere-positive {M}arkov chains.
\newblock {\em Statistics \& probability letters}, 22(4):333--338, 1995.

\bibitem{BiKoJu15}
A.~Bittracher, P.~Koltai, and O.~Junge.
\newblock Pseudogenerators of spatial transfer operators.
\newblock {\em SIAM Journal on Applied Dynamical Systems}, 14(3):1478--1517,
  2015.

\bibitem{BS13}
E.~M. Bollt and N.~Santitissadeekorn.
\newblock {\em Applied and Computational Measurable Dynamics}.
\newblock Society for Industrial and Applied Mathematics, 2013.

\bibitem{BoMu01}
C.~J. Bose and R.~Murray.
\newblock The exact rate of approximation in {U}lam's method.
\newblock {\em Discrete and Continuous Dynamical Systems}, 7(1):219--235, 2001.

\bibitem{BoMu06b}
C.~J. Bose and R.~Murray.
\newblock Dynamical conditions for convergence of a maximum entropy method for
  {F}robenius--{P}erron operator equations.
\newblock {\em Applied mathematics and computation}, 182(1):210--212, 2006.

\bibitem{BoMu06a}
C.~J. Bose and R.~Murray.
\newblock Minimum `energy' approximations of invariant measures for nonsingular
  transformations.
\newblock {\em Discrete and Continuous Dynamical Systems}, 14(3):597--615,
  2006.

\bibitem{BoMu07}
C.~J. Bose and R.~Murray.
\newblock Duality and the computation of approximate invariant densities for
  nonsingular transformations.
\newblock {\em SIAM Journal on Optimization}, 18(2):691--709, 2007.

\bibitem{BoGo12}
A.~Boyarsky and P.~Gora.
\newblock {\em Laws of Chaos: Invariant Measures and Dynamical Systems in One
  Dimension}.
\newblock Springer Science \& Business Media, 2012.

\bibitem{Boyd01}
J.~P. Boyd.
\newblock {\em Chebyshev and {F}ourier Spectral Methods}.
\newblock Dover Publications, Inc., 2nd edition, 2001.

\bibitem{BMM12}
M.~Budi{\v s}i{\'c}, R.~Mohr, and I.~Mezi{\'c}.
\newblock Applied {K}oopmanism.
\newblock {\em Chaos: An Interdisciplinary Journal of Nonlinear Science},
  22(4), 2012.

\bibitem{BuGr04}
H.-J. Bungartz and M.~Griebel.
\newblock Sparse grids.
\newblock {\em Acta Numerica}, 13:1--123, 2004.

\bibitem{Amber15}
D.~A. Case, J.~T. Berryman, R.~M. Betz, D.~S. Cerutti, T.~E. Cheatham, T.~A.
  Darden, R.~E. Duke, T.~J. Giese, H.~Gohlke, A.~W. Goetz, N.~Homeyer,
  S.~Izadi, P.~Janowski, J.~Kaus, A.~Kovalenko, T.~S. Lee, S.~LeGrand, P.~Li,
  T.~Luchko, R.~Luo, B.~Madej, K.~M. Merz, G.~Monard, P.~Needham, H.~Nguyen,
  H.~T. Nguyen, I.~Omelyan, A.~Onufriev, D.~R. Roe, A.~Roitberg,
  R.~Salomon-Ferrer, C.~L. Simmerling, W.~Smith, J.~Swails, R.~C. Walker,
  J.~Wang, R.~M. Wolf, X.~Wu, D.~M. York, and P.~A. Kollman.
\newblock {\em AMBER 2015}.
\newblock University of California, San Francisco, 2015.

\bibitem{Chekroun2014}
M.~D. Chekroun, J.~D. Neelin, D.~Kondrashov, J.~C. McWilliams, and M.~Ghil.
\newblock {Rough parameter dependence in climate models and the role of
  Ruelle--Pollicott resonances}.
\newblock {\em Proceedings of the National Academy of Sciences},
  111(5):1684--1690, 2014.

\bibitem{CU02}
G.~Chen and T.~Ueta, editors.
\newblock {\em Chaos in Circuits and Systems}.
\newblock World Scientific Series on Nonlinear Science, Series B, Volume 11.
  World Scientific, 2002.

\bibitem{DFJ00}
M.~Dellnitz, G.~Froyland, and O.~Junge.
\newblock The algorithms behind {GAIO} -- {S}et oriented numerical methods for
  dynamical systems.
\newblock In {\em Ergodic theory, analysis, and efficient simulation of
  dynamical systems}, pages 145--174. Springer, 2000.

\bibitem{DJ99}
M.~Dellnitz and O.~Junge.
\newblock On the approximation of complicated dynamical behavior.
\newblock {\em SIAM Journal on Numerical Analysis}, 36(2):491--515, 1999.

\bibitem{Ding98}
J.~Ding.
\newblock A maximum entropy method for solving {F}robenius--{P}erron operator
  equations.
\newblock {\em Applied Mathematics and Computation}, 93:155--168, 1998.

\bibitem{DiDuLi}
J.~Ding, Q.~Du, and T.-Y. Li.
\newblock High order approximation of the {F}robenius--{P}erron operator.
\newblock {\em Applied Mathematics and Computation}, 53:151--171, 1993.

\bibitem{DiLi91}
J.~Ding and T.-Y. Li.
\newblock Markov finite approximation of the {F}robenius--{P}erron operator.
\newblock {\em Nonlinear Analysis: Theory, Methods \& Applications},
  17:759--772, 1991.

\bibitem{DiZh96}
J.~Ding and A.~Zhou.
\newblock Finite approximations of {F}robenius--{P}erron operators. {A}
  solution of {U}lam's conjucture on multi-dimensional transformations.
\newblock {\em Physica D}, 92:61--68, 1996.

\bibitem{Fed69}
H.~Federer.
\newblock {\em Geometric measure theory}.
\newblock Springer New York, 1969.

\bibitem{Fr98}
G.~Froyland.
\newblock Approximating physical invariant measures of mixing dynamical
  systems.
\newblock {\em Nonlinear Analysis, Theory, Methods, \& Applications},
  32:831--860, 1998.

\bibitem{FGTW15}
G.~Froyland, C.~Gonz{\'a}lez-Tokman, and T.~M. Watson.
\newblock Optimal mixing enhancement by local perturbation.
\newblock 2015.
\newblock Preprint.

\bibitem{FGH14a}
G.~Froyland, G.~Gottwald, and A.~Hammerlindl.
\newblock A computational method to extract macroscopic variables and their
  dynamics in multiscale systems.
\newblock {\em SIAM Journal on Applied Dynamical Systems}, 13(4):1816--1846,
  2014.

\bibitem{FrJu15}
G.~Froyland and O.~Junge.
\newblock On fast computation of finite-time coherent sets using radial basis
  functions.
\newblock {\em Chaos}, 25(8), 2015.

\bibitem{FrJuKo13}
G.~Froyland, O.~Junge, and P.~Koltai.
\newblock Estimating long term behavior of flows without trajectory
  integration: The infinitesimal generator approach.
\newblock {\em SIAM Journal on Numerical Analysis}, 51(1):223--247, 2013.

\bibitem{FSvS14}
G.~Froyland, R.~M. Stuart, and E.~van Sebille.
\newblock How well-connected is the surface of the global ocean?
\newblock {\em Chaos: An Interdisciplinary Journal of Nonlinear Science},
  24(3), 2014.

\bibitem{Hal56}
P.~R. Halmos.
\newblock {\em Lectures on ergodic theory}, volume 142.
\newblock American Mathematical Soc., 1956.

\bibitem{Hopf54}
E.~Hopf.
\newblock The general temporally discrete {M}arkoff process.
\newblock {\em Journal of Rational Mechanics and Analysis}, 3(1):13--45, 1954.

\bibitem{Hub09}
P.~Huber.
\newblock D\"unngitter-{S}pektralmethoden zur {A}pproximation des
  {F}robenius--{P}erron-{O}perators.
\newblock Diploma thesis (in {G}erman), Technische Universit\"at M\"unchen,
  2009.

\bibitem{JSN14}
M.~R. Jovanovi{\'c}, P.~J. Schmid, and J.~W. Nichols.
\newblock Sparsity-promoting dynamic mode decomposition.
\newblock {\em Physics of Fluids}, 26(2), 2014.

\bibitem{JuKo09}
O.~Junge and P.~Koltai.
\newblock Discretization of the {F}robenius--{P}erron operator using a sparse
  {H}aar tensor basis: The {S}parse {U}lam method.
\newblock {\em SIAM Journal on Numerical Analysis}, 47:3464--3485, 2009.

\bibitem{KPPHD}
P.~Koltai.
\newblock {\em Efficient approximation methods for the global long-term
  behavior of dynamical systems -- {T}heory, algorithms and examples}.
\newblock PhD thesis, Technische Universit{\"a}t M{\"u}nchen, 2010.

\bibitem{Ko31}
B.~Koopman.
\newblock Hamiltonian systems and transformation in {H}ilbert space.
\newblock {\em Proceedings of the National Academy of Sciences of the United
  States of America}, 17(5):315, 1931.

\bibitem{Kre85}
U.~Krengel.
\newblock {\em Ergodic theorems}, volume~6 of {\em de {G}ruyter Studies in
  Mathematics}.
\newblock Walter de Gruyter \& Co., Berlin, 1985.

\bibitem{LaMa94}
A.~Lasota and M.~C. Mackey.
\newblock {\em Chaos, fractals, and noise: Stochastic aspects of dynamics},
  volume~97 of {\em Applied Mathematical Sciences}.
\newblock Springer, 2nd edition, 1994.

\bibitem{Li76}
T.-Y. Li.
\newblock Finite approximation for the {F}robenius--{P}erron operator. {A}
  solution to {U}lam's conjecture.
\newblock {\em Journal of Approximation Theory}, 17:177--186, 1976.

\bibitem{MaSt02}
J.~C. Mattingly and A.~M. Stuart.
\newblock Geometric ergodicity of some hypo-elliptic diffusions for particle
  motions.
\newblock {\em Markov Process. Related Fields}, 8(2):199--214, 2002.

\bibitem{MeTw12}
S.~P. Meyn and R.~L. Tweedie.
\newblock {\em Markov chains and stochastic stability}.
\newblock Springer Science \& Business Media, 2012.

\bibitem{MurrPhD}
R.~Murray.
\newblock {\em Discrete approximation of invariant densities}.
\newblock PhD thesis, University of Cambridge, 1997.

\bibitem{Mur04}
R.~Murray.
\newblock Optimal partition choice for invariant measure approximation for
  one-dimensional maps.
\newblock {\em Nonlinearity}, 17(5):1623--1644, 2004.

\bibitem{NoNu13}
F.~No{\'e} and F.~N{\"u}ske.
\newblock A variational approach to modeling slow processes in stochastic
  dynamical systems.
\newblock {\em Multiscale Modeling \& Simulation}, 11(2):635--655, 2013.

\bibitem{NKPMN14}
F.~N\"uske, B.~G. Keller, G.~P\'erez-Hern\'andez, A.~S. J.~S. Mey, and
  F.~No\'e.
\newblock Variational approach to molecular kinetics.
\newblock {\em Journal of Chemical Theory and Computation}, 10(4):1739--1752,
  2014.

\bibitem{NSVN15}
F.~N\"uske, R.~Schneider, F.~Vitalini, and F.~No\'e.
\newblock Variational tensor approach for approximating the rare-event kinetics
  of macromolecular systems.
\newblock {\em The Journal of Chemical Physics}, 144(5), 2016.

\bibitem{ober2015multiobjective}
S.~Ober-Bl{\"o}baum and K.~Padberg-Gehle.
\newblock Multiobjective optimal control of fluid mixing.
\newblock {\em PAMM}, 15(1):639--640, 2015.

\bibitem{Orn70}
D.~Ornstein.
\newblock Bernoulli shifts with the same entropy are isomorphic.
\newblock {\em Advances in Mathematics}, 4(3):337--352, 1970.

\bibitem{PDHSM04}
R.~Preis, M.~Dellnitz, M.~Hessel, C.~Sch{\"u}tte, and E.~Meerbach.
\newblock Dominant paths between almost invariant sets of dynamical systems.
\newblock {DFG} {S}chwerpunktprogramm 1095, {T}echnical {R}eport 154, 2004.

\bibitem{SS08}
P.~Schmid and J.~Sesterhenn.
\newblock Dynamic {M}ode {D}ecomposition of numerical and experimental data.
\newblock In {\em 61st Annual Meeting of the APS Division of Fluid Dynamics}.
  American Physical Society, 2008.

\bibitem{PyMOL10}
{Schr\"odinger, LLC}.
\newblock The {PyMOL} molecular graphics system, {V}ersion~1.7.4, 2014.

\bibitem{Sch99}
C.~Sch{\"u}tte.
\newblock Conformational dynamics: Modelling, theory, algorithm, and
  application to biomolecules, 1999.
\newblock Habilitation Thesis.

\bibitem{SS13}
C.~Sch\"utte and M.~Sarich.
\newblock {\em Metastability and Markov State Models in Molecular Dynamics:
  Modeling, Analysis, Algorithmic Approaches}.
\newblock Number~24 in Courant Lecture Notes. American Mathematical Society,
  2013.

\bibitem{Sin59}
Y.~G. Sinai.
\newblock On the notion of entropy of dynamical systems.
\newblock In {\em Doklady Akademii Nauk}, volume 124, pages 768--771, 1959.

\bibitem{TaLuLuDi15}
A.~Tantet, V.~Lucarini, F.~Lunkeit, and H.~A. Dijkstra.
\newblock Crisis of the chaotic attractor of a climate model: a transfer
  operator approach.
\newblock {\em Preprint, arXiv:1507.02228}, 2015.

\bibitem{TvdBD15}
A.~Tantet, F.~R. {van der Burgt}, and H.~A. Dijkstra.
\newblock An early warning indicator for atmospheric blocking events using
  transfer operators.
\newblock {\em Chaos}, 25(3), 2015.

\bibitem{TRLBK13}
J.~H. Tu, C.~W. Rowley, D.~M. Luchtenburg, S.~L. Brunton, and J.~N. Kutz.
\newblock On {D}ynamic {M}ode {D}ecomposition: {T}heory and {A}pplications.
\newblock {\em ArXiv e-prints}, 2013.

\bibitem{Ulam60}
S.~M. Ulam.
\newblock {\em A Collection of Mathematical Problems}.
\newblock Interscience Publisher NY, 1960.

\bibitem{VMS10}
U.~Vaidya, P.~G. Mehta, and U.~V. Shanbhag.
\newblock Nonlinear stabilization via control {L}yapunov measure.
\newblock {\em IEEE Transactions on Automatic Control}, 55(6):1314--1328, 2010.

\bibitem{WKR14}
M.~O. Williams, I.~G. Kevrekidis, and C.~W. Rowley.
\newblock A data-driven approximation of the {K}oopman operator: Extending
  dynamic mode decomposition.
\newblock {\em ArXiv e-prints}, 2014.

\bibitem{WRK14}
M.~O. Williams, C.~W. Rowley, and I.~G. Kevrekidis.
\newblock A kernel-based approach to data-driven {K}oopman spectral analysis.
\newblock {\em ArXiv e-prints}, 2014.

\bibitem{WRR15}
M.~O. Williams, I.~I. Rypina, and C.~W. Rowley.
\newblock Identifying finite-time coherent sets from limited quantities of
  {L}agrangian data.
\newblock {\em Chaos}, 25(8), 2015.

\end{thebibliography}

\begin{appendix}

\section{Adjoint EDMD}
\label{app:ad_EDMD}

\subsection*{Notation.}

We will use the same notation as in the main text. More precisely, let us use the set of (linearly independent), piecewise continuous basis functions (dictionary)
\[
\mathbb{D} = \{\psi_1,\ldots,\psi_k\},\qquad \psi_i:\R^d\to\R,\ i=1,\ldots,k,
\]
and let us denote~$\mathbb{V} = \text{span}(\mathbb{D})$. For every~$f\in\R^k$, we define~$\bar{f} = \sum_{i=1}^kf_i\psi_i\in\mathbb{V}$. We also identify a linear operator~$A:\mathbb{V}\to\mathbb{V}$ with its matrix representation~$A\in\R^{k\times k}$ with respect to the basis~$\mathbb{D}$. Here we mean multiplication from the left, i.e.~$Kf$ is identified with~$K\bar{f}$. Further, let
\[
\Psi = \begin{pmatrix}
\psi_1\\ \vdots\\ \psi_k
\end{pmatrix}\,,
\]
a vector-valued function, and for sets of points (collected column-wise into a~$d\times m$ matrix)
\[
X = [x_1\ x_2\ \ldots\ x_m],\qquad Y = [y_1\ y_2\ \ldots\ y_m]
\]
define
\[
\Psi_X = [\Psi(x_1)\ \Psi(x_2)\ \ldots\ \Psi(x_m)],\qquad \Psi_Y = [\Psi(y_1)\ \Psi(y_2)\ \ldots\ \Psi(y_m)]\,.
\]

\subsection*{Scalar products.}

Given~$f,g\in\R^k$ and some positive measure~$\mu$, such that~$|\int \psi_i\psi_j\,d\mu|<\infty$ for all~$i,j=1,\ldots,k$, we wish to express the~$\mu$-weighted $L^2$ scalar products of elements of~$\mathbb{V}$. To this end, we compute
\[
\langle \bar{f},\bar{g}\rangle_{\mu} = \int \bar{f}\bar{g}\,d\mu = \sum_{i,j=1}^k f_ig_j \int \psi_i\psi_j\,d\mu = f^T S g\,,
\]
where~$S\in\R^{k\times k}$ with~$S_{ij} = \int \psi_i\psi_j\,d\mu$. Since~$\mu$ is positive,~$S$ is symmetric positive definite, hence invertible.

\subsection*{Adjoint operator.}

With this, we are ready to express the adjoint~$A^*$ of any (linear) operator~$A:\mathbb{V}\to\mathbb{V}$ with respect to the scalar product~$\langle \cdot,\cdot\rangle_{\mu}$.
By successive reformulations of the defining equation for the adjoint, we obtain
\[
\begin{aligned}
\langle A\bar{f},\bar{g}\rangle_{\mu} &= \langle \bar{f},A^*\bar{g}\rangle_{\mu}&\forall \bar{f},\bar{g}\in\mathbb{V}, \\
&\Updownarrow \\
\langle \sum_{i=1}^k(Af)_i\psi_i,\sum_{i=1}^k g_i\psi_i\rangle_{\mu} &= \langle \sum_{i=1}^kf_i\psi_i,\sum_{i=1}^k (A^*g)_i\psi_i\rangle_{\mu} &\forall f,g\in\R^k,\\
&\Updownarrow \\
f^TA^TSg &= f^TSA^*g &\forall f,g\in \R^k.
\end{aligned}
\]
Thus,
\begin{equation}
A^* = S^{-1}A^TS\,.
\label{eq:adj_op}
\end{equation}
\begin{remark}
From~\eqref{eq:adj_op} we can see that~$A^T$ represents the adjoint of~$A$ if~$S$ is a multiple of the identity matrix, implying that the basis functions are orthogonal with respect to~$\langle\cdot,\cdot\rangle_{\mu}$. This is the case for Ulam's method, given the boxes have all the same measures.
\end{remark}

\subsection*{The Perron--Frobenius operator.}

Let~$\Phi:\R^d\to\R^d$ be some dynamical system. The following properties hold also, if~$\Phi$, such as the basis functions and the measure~$\mu$ are restricted to some set~$\mathbb{X}$.

Recall equations~\eqref{eq:tdf} and~\eqref{eq:Koopmanop}, stating that the \emph{Perron--Frobenius operator}~$\mathcal{P}_{\mu}:L^1\to L^1$ with respect to the measure~$\mu$ is (uniquely) defined by
\[
\int_{\mathbb{A}} \mathcal{P}_{\mu}h\,d\mu = \int_{\Phi^{-1}(\mathbb{A})} h\,d\mu,\quad\text{for all measurable }\mathbb{A}\,,
\]
and the \emph{Koopman operator}~$\mathcal{K}:L^{\infty}\to L^{\infty}$ is defined by
\[
\mathcal{K}h = h\circ\Phi\,,
\]
respectively. They satisfy the duality relation
\[
\langle \mathcal{P}_{\mu}h_1,h_2\rangle_{\mu} = \langle h_1,\mathcal{K}h_2\rangle_{\mu}\quad\forall h_1\in L^1,\, h_2\in L^{\infty}\,.
\]

We have seen in section~\ref{ssec:EDMD}, that if the data points satisfy~$y_i = \Phi(x_i)$, $i=1,\ldots,m$, then~$K$, with~$K^T = \Psi_Y\Psi_X^+$, is a data-based approximation of the Koopman operator. More precisely, in the infinite-data limit~$m\to\infty$,~$x_i\sim\mu$, the operator~$K$ converges to a Galerkin approximation of~$\mathcal{K}$ on~$\mathbb{V}$ with respect to~$\langle\cdot,\cdot\rangle_{\mu}$. Using~\eqref{eq:EDMDtoGalerkin}, we can also conclude that
\[
\frac{1}{m} \Psi_X\Psi_X^T \to S \text{ as }m\to\infty\,,
\]
where~$S$ is the symmetric positive definite weight matrix from above. This suggests, using~\eqref{eq:adj_op}, that if there is a sufficient amount of data points at hand, then we can approximate the Galerkin projection of the Perron--Frobenius operator~$\mathcal{P}_{\mu}$ to~$\mathbb{V}$ by
\begin{equation}
P_{\mu} = S^{-1}K^T S = (\Psi_X\Psi_X^T)^{-1}\Psi_Y\Psi_X^+ (\Psi_X\Psi_X^T) = (\Psi_X\Psi_X^T)^{-1}\Psi_Y\Psi_X^T\,.
\label{eq:FPO}
\end{equation}
The same matrix representation has been obtained in equation~\eqref{eq:M_P}, by a different consideration. Note also, that if one can compute the matrix~$S$ with~$S_{ij} = \int\psi_i\psi_j\,d\rho$ with respect to a different measure~$\rho$, the Perron--Frobenius operator with respect to~$\rho$ can be approximated as well, one is not restricted to use the empirical distribution~$\mu$ of the data points.

\begin{remark}
All these considerations can be extended to the case where the dynamics~$\Phi$ is non-deterministic.
\end{remark}

\section{On the ergodic behavior of one-step pairs}

We will need the result of this section, equation~\eqref{eq:ergTHMcomposite}, in the following section.

Let the non-deterministic dynamical system~$ \bm\Phi $ be given with transition density function~$ k $, that is,
\begin{equation*}
    \prob\left(\bm\Phi(x)\in\mathbb{A}\right) = \int_{\mathbb{A}}k(x,y)\,d\mu(y),\quad \mathbb{A}\in\mathfrak{B},
\end{equation*}
for a.e.~$y\in\mathbb{X}$. Further, let~$f$ denote the unique invariant density of~$\bm\Phi$,
\begin{equation*}
    \int f(x)k(x,y)\,d\mu(x) = f(y)\quad\text{for a.e. }y\in\mathbb{X},
\end{equation*}
with respect to which~$\bm\Phi$ is geometrically ergodic. Geometric ergodicity of the Langevin process~\eqref{eq:Langevin} has been established in~\cite{MaSt02}.

For $\phi,\psi\in L^2(\mathbb{X})$ we wish to determine the ergodic limit
\begin{equation*}
    \lim_{N\to\infty}\frac{1}{N}\sum_{n=0}^{N-1} \phi\big(\bm\Phi^n(x)\big)\psi\big(\bm\Phi^{n+1}(x)\big)\,.
\end{equation*}
To this end, we consider the non-deterministic dynamical system $\bm\Psi : \mathbb{X} \times \mathbb{X} \to \mathbb{X} \times \mathbb{X} $ with
\begin{equation*}
    \bm\Psi : \begin{pmatrix}
    x \\ y
    \end{pmatrix} \mapsto \begin{pmatrix}
    y \\ \bm\Phi(y)
    \end{pmatrix}\,.
\end{equation*}
In order to find the transition density function of~$\bm\Psi$, note that
\begin{equation*}
    \prob\big(\bm\Psi(x,y)\in\mathbb{A}\times\mathbb{B}\big) = \mathds{1}_{\mathbb{A}}(y)\int_{\mathbb{B}}k(y,z)\,d\mu(z) = \int_{\mathbb{A}\times\mathbb{B}} \delta_y(u)k(u,z)\,d\mu(u) \, d\mu(z)\,,
\end{equation*}
yielding $k_{\bm\Psi}((x,y),(u,z)) = \delta_y(u) k(u,z)$ as the transition density function of~$\bm\Psi$. From this we immediately find its invariant density.

\begin{lemma} \label{lem:invdenscomposite}
The density~$f(x)k(x,y)$ is invariant under~$\bm\Psi$.
\end{lemma}
\begin{proof}
Direct computation shows
\begin{align*}
\iint f(x)k(x,y) & k_{\bm\Psi}((x,y),(u,z))\,d\mu(x) \, d\mu(y) \\
    & = \iint f(x)k(x,y)  \delta_y(u)k(u,z)\,d\mu(x) \, d\mu(y) \\
    & = k(u,z)\int f(x)\int k(x,y)  \delta_y(u)\,d\mu(y) \, d\mu(x) \\
    & = k(u,z)\int f(x) k(x,u) \, d\mu(x)\\
    & = f(u) k(u,z),
\end{align*}
the last equality following from the invariance of~$f$ under~$\bm\Phi$.
\end{proof}

Geometric ergodicity of~$\bm\Phi$ with respect to~$f$ implies ergodicity of~$\bm\Psi$ with respect to~$f(x)k(x,y)$. Thus, for~$\zeta\in L^2(\mathbb{X}\times\mathbb{X})$ we have
\begin{equation*}
    \lim_{N\to\infty}\frac{1}{N}\sum_{n=0}^{N-1} \zeta\big(\bm\Psi^n(x)\big) = \iint \zeta(x,y)f(x)k(x,y)\,d\mu(x) \, d\mu(y)\,.
\end{equation*}
With $\zeta(x,y) = \phi(x)\psi(y)$ this implies
\begin{equation} \label{eq:ergTHMcomposite}
\begin{split}
    \lim_{N\to\infty}\frac{1}{N}&\sum_{n=0}^{N-1} \phi\big(\bm\Phi^n(x)\big)\psi\big(\bm\Phi^{n+1}(x)\big) \\
    & = \lim_{N\to\infty}\frac{1}{N}\sum_{n=0}^{N-1} \zeta\big(\bm\Psi^n(x)\big) \\
    & = \iint \zeta(x,y)f(x)k(x,y)\,d\mu(x) \, d\mu(y) \\
    & = \iint \phi(x)f(x)\psi(y)k(x,y)\,d\mu(x)\, d\mu(y) \\
    & = \int \psi(y) \int (\phi(x)f(x)) k(x,y)\,d\mu(x) \, d\mu(y) \\
    & = \int \psi\, \P(\phi f)\,d\mu\,,
\end{split}
\end{equation}
where the last equality follows from~\eqref{eq:FPOnondet}, the definition of the Perron--Frobenius operator.

\section{EDMD for the reduced spatial transfer operator} \label{app:EDMD spatial essential}

We shall first discuss the restriction of the spatial transfer operator, introduced in~\eqref{eq:spatial}, to a collection of coordinates which we assume to be sufficient to describe the metastable behavior of the system. Let~$\xi:\qstate\to\mathbb{U}\subset\R^r$ be a smooth, possibly nonlinear mapping of the configuration variable~$q$ to these so-called \emph{essential} (or \emph{reduced}) coordinates. For instance, in case of $n$-butane in Section~\ref{ssec:butane} we have~$r=1$ and~$\xi$ describes the mapping~$q\mapsto\varphi$ given implicitly by~\eqref{eq:dihedral}. Let~$\xi$ have the property that for every regular value~$z\in\mathbb{U}$ of~$\xi$,
\begin{equation*}
    \mathbb{M}_z := \{q\in\qstate\,\vert\, \xi(q) = z\}\subset\qstate
\end{equation*}
is a smooth, codimension~$r$ manifold. We suppose that~$\xi$ is a physically relevant observable of the dynamics, e.g.\ a reaction coordinate.

To define the spatial transfer operator for the essential coordinates, we need a nonlinear variant of Fubini's theorem, the so-called \emph{coarea formula}~\cite[Section 3.2]{Fed69}. For an integrable function~$h:\qstate\to\R$ it holds
\begin{equation} \label{eq:coarea}
    \int_{\qstate} h(q)\,dq = \int_{\mathbb{U}}\left(\int_{\mathbb{M}_z}h G\,d\sigma_z\right)\,dz,
\end{equation}
where~$G(q) = \left|\det \nabla\xi^T\nabla\xi\right|^{-1/2}$ is the Gramian, and~$d\sigma_z$ denotes the Riemannian volume element on~$\mathbb{M}_z$. It follows that the (marginal) canonical density for the observable~$\xi$ is
\begin{equation*}
     f_{\mathbb{U}}(z) = \iint_{\mathbb{M}_z\times\pstate} f_{\rm can}G\,d\sigma_z\,dp\,.
\end{equation*}
Thus, the spatial transfer operator for the essential coordinates given by~$\xi$ reads as
\begin{equation} \label{eq:spatial_ess}
    \Sp^t_{\rm ess}w(z) = \frac{1}{f_{\mathbb{U}}(z)}\iint_{\mathbb{M}_z\times\pstate} \P^t_{\rm Lan}(f_{\rm can}\cdot w\circ\xi)G\,d\sigma_z\,dp\,,
\end{equation}
for $w\in L^2(\mathbb{U},\mu_{\mathbb{U}})$, with~$d\mu_{\mathbb{U}}(z) = f_{\mathbb{U}}(z)dz$.

To see what EDMD does with the molecular trajectory data, we have to consider the limit
\begin{equation*}
    \lim_{N\to\infty}\frac{1}{N}\sum_{n=0}^{N-1}\phi(\xi(q_n))\psi(\xi(q_{n+1})),
\end{equation*}
where $q_0,q_1,q_2,\ldots$ are the positional coordinates of the time-$t$-sampled simulation, and $\phi,\psi:\mathbb{U}\to\R$ are basis functions.
We know that the Langevin dynamics is ergodic with respect to the canonical density~\cite{MaSt02}, hence~\eqref{eq:ergTHMcomposite} yields
\begin{align*}
    \lim_{N\to\infty}\frac{1}{N} \sum_{n=0}^{N-1}\phi(\xi(q_n))\psi(\xi(q_{n+1}))
    & = \iint_{\qstate\times\pstate}\psi(\xi(q))\P^t_{\rm Lan}\left(f_{\rm can}\cdot \phi\circ\xi\right)(q,p)\,dq \, dp \\
    & = \int_{\qstate}\psi(\xi(q))\int_{\pstate}\P^t_{\rm Lan}\left(f_{\rm can}\cdot \phi\circ\xi\right)(q,p)\,dq \, dp \\
    & \stackrel{\eqref{eq:coarea}}{=} \int_{\mathbb{U}}\psi(z)\iint_{\mathbb{M}_z\times\pstate}\P^t_{\rm Lan}\left(f_{\rm can}\cdot \phi\circ\xi\right)(q,p)G(q)\,dp \, d\sigma_z \, dz \\
    & \stackrel{\eqref{eq:spatial_ess}}{=} \int_{\mathbb{U}} \psi(z) f_{\mathbb{U}}(z) S^t_{\rm ess}\phi(z)\,dz \\
    & = \innerprod{\psi}{S^t_{\rm ess}\phi}_{\mu_{\mathbb{U}}}.
\end{align*}
Due to ergodicity it follows also
\begin{align*}
    \lim_{N\to\infty}\frac{1}{N}\sum_{n=0}^{N-1}\phi(\xi(q_n))\psi(\xi(q_n))
    & = \iint_{\qstate\times\pstate}\psi(\xi(q))\phi(\xi(q)) f_{\rm can}(q,p)\,dq \, dp \\
    & \stackrel{\eqref{eq:coarea}}{=} \int_{\mathbb{U}}\psi(z)\phi(z)\iint_{\mathbb{M}_z\times\pstate} f_{\rm can}(q,p)G(q)\,dp\,d\sigma_z\,dz \\
    & = \innerprod{\psi}{\phi}_{\mu_{\mathbb{U}}}.
\end{align*}
Comparing with~\eqref{eq:EDMDtoGalerkin}, we thus see that EDMD converges in the infinite-data limit to a Galerkin projection in~$L^2(\mathbb{U},\mu_{\mathbb{U}})$ of the spatial transfer operator for the essential coordinates given by~$\xi$.

\section{Derivation of the EDMD-discretized Koopman operator}
\label{sec:Derivation of EDMD}

Let the finite dictionary~$\mathbb{D} = \{\psi_1,\ldots,\psi_k\}$ of piecewise continuous functions be given, and define~$\mathbb{V}$ to be the linear space spanned by~$\mathbb{D}$. We will give a step-by-step derivation of the matrix representation of the EDMD-discretized Koopman operator~$ K: \mathbb{V} \to \mathbb{V} $ with respect to the basis~$ \mathbb{D} $. Let us denote also with~$ K \in \R^{k\times k} $ this matrix representation, and note that the matrix~$ K $ acts by multiplication from the left, i.e.\ if the vector~$ c \in \R^k $ represents the function~$ \sum_i c_i \psi_i $, then~$ K \, c $ represents its image under the discrete Koopman operator. Recall that~$ \psi : \mathbb{X} \to \R^k $ denotes the column-vector valued function with~$ [\psi(x)]_i = \psi_i(x) $.

Now, EDMD is an over-determined Petrov--Galerkin method~\eqref{eq:PetrovGalerkinLS},
\begin{equation} \label{eq:EDMD_LS}
    \sum_{\ell=1}^l\sum_{i=1}^k \innerprod{\delta_{x_{\ell}}}{\mathcal{K}\psi_i - K\psi_i}^2 = \min!\,,
\end{equation}
where~$x_1\ldots,x_l$ are the initial data points and~$y_1,\ldots,y_l$ denote their images under the dynamics. If there was just one single data point~$x_{\ell}$, we would like to find a matrix~$K$ satisfying the equation
\begin{equation*}
    \innerprod{\delta_{x_{\ell}}}{\mathcal{K}\Big(\sum_i c_i \, \psi_i\Big)} =
    \innerprod{\delta_{x_{\ell}}}{\sum_i\Big(\sum_j K_{ij} \, c_j\Big)\psi_i}
\end{equation*}
for every~$c\in\R^k$. Rearranging the terms and using $\mathcal{K}\psi_i(x_{\ell}) = \psi_i(y_{\ell})$ yields
\begin{equation*}
    \sum_i c_i \, \psi_i(y_{\ell}) = \sum_j c_j \sum_i K_{ij} \, \psi_i(x_{\ell})\,,
\end{equation*}
or, in vectorial notation,~$ c^T \, \psi(y_{\ell}) = c^T \, K^T \psi(x_{\ell})$. Since this has to hold true for every~$ c \in \R^k $, we have~$ \psi(y_{\ell}) = K^T \psi(x_{\ell})$. From this it follows by putting the column vectors~$ \psi(x_{\ell}) $ and~$ \psi(y_{\ell}) $ side-by-side for multiple data points~$ x_{\ell} $ to form the matrices~$ \Psi_X $ and~$ \Psi_Y $, respectively, that~\eqref{eq:EDMD_LS} is equivalent with
\begin{equation*}
    \norm{\Psi_Y - K^T \Psi_X}_F^2 = \min!,
\end{equation*}
where~$ \norm{\cdot}_F $ denotes the Frobenius norm. Thus, EDMD can be viewed as a DMD of the transformed data $ \Psi_X $ and $ \Psi_Y $. The solution of the minimization problem is given by
\begin{equation*}
    K^T = \Psi_Y \Psi_X^+,
\end{equation*}
where~$ \Psi_X^+ $ is the pseudoinverse of~$ \Psi_X $.
\end{appendix}

\end{document}